\documentclass[11pt]{amsart}
\usepackage{amssymb,mathrsfs,graphicx,enumitem}
\usepackage{colortbl}
\definecolor{black}{rgb}{0.0, 0.0, 0.0}
\definecolor{red}{rgb}{1.0, 0.5, 0.5}

\title[   ]{Existence and stability of planar shocks of viscous scalar conservation laws with space-periodic flux}

\author[Dalibard]{Anne-Laure Dalibard}
\address[Anne-Laure Dalibard]{Sorbonne Universit\'es, UPMC Univ Paris 06, CNRS, UMR 7598, Laboratoire Jacques-Louis Lions, 4, place Jussieu 75005, Paris, France.}
\email{dalibard@ljll.math.upmc.fr}

\author[Kang]{Moon-Jin Kang}
\address[Moon-Jin Kang]{\newline Department of Mathematics, \newline The University of Texas at Austin, Austin, TX 78712, USA}
\email{moonjinkang@math.utexas.edu}

\newtheorem{lemma}{Lemma}[section]

\newtheorem{proposition}{Proposition}[section]
\newtheorem{remark}{Remark}[section]

\newtheorem{definition}{Definition}[section]

\newcommand{\bbr}{\mathbb R}

\newcommand{\bbz}{\mathbb Z}

\newcommand{\bbt} {\mathbb T}

\newcommand{\na}{\nabla}
\numberwithin{figure}{section}
\newcommand{\beq}{\begin{equation}}
\newcommand{\eeq}{\end{equation}}
\newcommand{\bsp}{\begin{split}}
\newcommand{\esp}{\end{split}}

\newcommand{\dv}{\mathrm{div}}

\newcommand{\sgn}{{\text{\rm sgn}}}

\newcommand{\RR}{{\mathbb R}}

\newcommand{\NN}{{\mathbb N}}
\newcommand{\ZZ}{{\mathbb Z}}
\newcommand{\TT}{{\mathbb T}}

\def\eps{\varepsilon }

\newcommand\adots{\mathinner{\mkern2mu\raise1pt\hbox{.}
\mkern3mu\raise4pt\hbox{.}\mkern1mu\raise7pt\hbox{.}}}

\newtheorem{theo}{Theorem}[section]
\newtheorem{prop}[theo]{Proposition}
\newtheorem{cor}[theo]{Corollary}
\newtheorem{lem}[theo]{Lemma}

\newtheorem{rem}[theo]{Remark}

\newcommand{\be}{\begin{equation}}
\newcommand{\ee}{\end{equation}}
\newcommand{\ba}{\begin{aligned}}
\newcommand{\ea}{\end{aligned}}
\newcommand{\p}{\partial}

\def\charf {\mbox{{\text 1}\kern-.30em {\text l}}}

\setlist[enumerate,1]{leftmargin=0cm}
\setlist[enumerate,1]{itemindent=0cm}
\newcommand{\mean}[1]{\left\langle #1\right\rangle}
\begin{document}

\date{\today}


\thanks{\textbf{Acknowledgment.} M.-J. Kang was supported by the Foundation Sciences Math$\acute{\mbox{e}}$matiques de Paris as a postdoctoral fellowship, and by an AMS-Simons Travel Grant.
}

\begin{abstract}
The goal of this paper is to prove the existence and stability of shocks for viscous scalar conservation laws with space periodic flux, in the multi-dimensional case. Such a result had been proved by the first author in one space dimension, but the extension to a multi-dimensional setting makes the existence proof non-trivial. We construct approximate solutions by restricting the size of the domain and then passing to the limit as the size of the domain goes to infinity. One of the key steps is a ``normalization'' procedure, which ensures that the limit objects obtained by the approximation scheme are indeed shocks. The proofs rely on elliptic PDE theory rather than ODE arguments as in the 1d case.  Once the existence of shocks is proved, their stability follows from classical arguments based on the theory of dynamical systems.

\end{abstract}
\maketitle \centerline{\date}


\section{Introduction and Main results}
In this article, we aim to show the existence and large time stability of multidimensional planar shock fronts of viscous scalar conservation laws with space-periodic flux: 
\begin{align}
\begin{aligned} \label{main}
&\partial_t u + \sum_{i=1}^N\partial_{x_i} A_i(x,u) = \Delta_x u, \quad t>0,~x\in \bbr^N,\\
&u(0,x) = u_0(x)
\end{aligned}
\end{align}
where the flux functions $A_i : \bbt^N \times \bbr \rightarrow \bbr^N$ are assumed to be periodic with respect to the spatial variable $x$.\\

The issues in the case of one dimension $N=1$ have been treated by the first author in \cite{Dalibard-indiana}, and therefore our goal is to tackle these issues in the multidimensional case ($N \geq 2$).\\
{ 
 When the flux $A$ is homogeneous, i.e. when $A$ does not depend on $x$, a planar shock wave is a special solution of \eqref{main} of the form $u(t,x)=U(x\cdot \nu - ct)$, for some $c\in \RR, \nu \in \mathbb S^{N-1}$, $U \in L^\infty(\RR^N)$, and with $\lim_{y\to \pm \infty } U(y)= U_\pm$, for some constants $U_+, U_- \in \RR$. The profile $U$ is easily found thanks to simple ODE theory together with Rankine-Hugoniot condition. But the stability of planar shock fronts is a challenging issues. Stability for a small perturbation of multidimensional planar shocks has been shown by Goodman \cite{Go}, Hoff and Zumbrun \cite{HZ}, and the second author, Vasseur and Wang \cite{K-V-W}. In one-dimensional case, Freist$\ddot{\mbox{u}}$hler and Serre \cite {FS} proved $L^1$-stability for any $L^1$-perturbation. Recently, the second author and Vasseur \cite{KV} have shown contraction for any $L^2$-perturbation.} \\
But when $A$ depends on the space variable,  the constants are no longer stationary solutions of \eqref{main} in general, and thus cannot be end states of planar shocks. Therefore we first introduce a family of periodic stationary solutions of \eqref{main}, which will play the role of constant solutions in the homogeneous case. These solutions were introduced in \cite{Dalibard-jmpa}.

\begin{proposition}[Existence of periodic stationary solutions of \eqref{main}, see \cite{Dalibard-jmpa}]\label{prop-exist} $ $

\noindent Let $A\in W^{1,\infty}_{loc} (\bbt^N\times \bbr)^N$. Assume that there exists $C_0>0$, and $m\in [0, \infty), n \in [0, \frac{N+2}{N-2})$ when $N>2$ such that for all $(x,v)\in  \bbt^N \times \bbr$,
\begin{subequations}
\begin{align}
&|\partial_v A_i(x,v)|\le C_0 (1+|v|^m),\quad 1\le i\le N, \label{growth-1}\\
&|\mbox{div}_x A(x,v)|\le C_0 (1+|v|^n). \label{growth-2}
\end{align}  
\end{subequations}
Assume as well that one of the following three conditions holds:
\begin{align}
\begin{aligned} \label{cons}
&i)~ m=0 \quad\mbox{or} \\
&ii)~0\le n< 1\quad \mbox{or}\\
&iii)~\Big(n<\min(\frac{N+2}{N},2)\quad \mbox{and}\quad\exists p_0~\mbox{s.t.}~ \sum_{i=1}^N\partial_{x_i} A_i(x,p_0)\equiv0\Big).
\end{aligned}
\end{align}
Then for each $p\in\bbr$, there exists a unique periodic solution $v(\cdot, p)\in H^1(\bbt^N)$ of the equation
\beq\label{cell}
- \Delta_x v(x,p) + \mbox{div}_{x} A(x,v(x,p)) =0,\quad <v(\cdot,p)>=p.
\eeq

\end{proposition}
In the above proposition and throughout the article, the brackets $\mean{\cdot}$ denote the average value of a $\bbt^N$-periodic function.

We list below further properties of the functions $v(x, p)$ (see Proposition \ref{prop:v-properties}). We also define the averaged - or homogenized - flux $\bar A$ by
$$
\bar A(p):=\mean{A(\cdot, v(\cdot, p))} \quad \forall p\in \RR.
$$

We are now ready to define stationary (or standing) planar shocks. 
\begin{definition}
  A stationary planar viscous shock of \eqref{main} with periodic end states is a function $\bar U\in  H^1_{loc}(\bbr^N)$ which is a stationary solution of \eqref{main}, periodic in the variables $x_1, \cdots x_{k-1}$, $x_{k+1}, \cdots x_N$ for some $k\in \{1,\cdots, N\}$, and such that there exist $p_+, p_-\in \RR$ with $p_+\neq p_-$ such that 
\beq\label{shock-b}
\lim_{x_k\to\pm\infty} \left( \bar{U} (x) - v(x,p_{\pm})\right)=0\quad\mbox{in}~ L^{\infty}(\bbt^{N-1}),
\eeq
Such a function is called a stationary shock of \eqref{main} with end states $v(\cdot, p_\pm)$, or  a stationary shock of \eqref{main} connecting $v(\cdot, p_-)$ to $v(\cdot, p_+)$.
\label{def:shock}

\end{definition}

\begin{remark}
Notice that because of the periodicity of the flux and of the stationary states, we only consider shocks in the directions $e_1, \cdots, e_N$, (i.e. in the directions of the canonical basis in $\RR^N$), and not in any direction $\nu\in\mathbb S^{N-1}$ as in the homogeneous case. Indeed, if we take an arbitrary direction $\nu$ and look for a shock such that $U(x\cdot \nu, x^\bot) - U_\pm (x) \to 0$ as $x\cdot \nu \to \pm \infty$, where $x^\bot\cdot \nu=0$, then in general the asymptotic states $U_\pm$ are not periodic solutions of \eqref{main}, but quasi-periodic solutions. Therefore a first step would be to study problems of the type
$$
-\Delta v + \dv \tilde A (x, v)=0
$$
where the flux $\tilde A$ is quasi-periodic in its first variable and the function $v$ is sought as periodic. This is expected to be much more difficult than in the periodic case, due to the lack of compactness and to the non-linearity. Such questions go beyond the scope of this paper, and thus we focus on periodic end states only.\\
Moreover, without loss of generality, we focus on the case when $k=1$ in the rest of the paper.
\end{remark}
{
The stationary shocks in Definition \ref{def:shock} can be viewed as a spatial transition front in a space-periodic environment. The spatial transition fronts arising in various (periodic) heterogeneities have also received a lot of attention in the reaction-diffusion community. In particular, the existence of spatial transition waves for one-dimensional space-heterogenous reaction-diffusion equation has been proved by Xin \cite{Xin} and Berestycki and Hamel \cite{BH-1}, and by Nolen and Ryzhik \cite{NR} and Mellet, Raquejoffre and Sire \cite{MRS} for ignition-type equation. These results have been extended by Zlatos \cite{Z} to multidimensional case of the cylindrical domain $\bbr\times\bbt^{N-1}$. We also refer to \cite{BH-2,BH-3} for a generalization of the notion of the transition fronts, whereas non-existence of such waves has been studied by Nadin \cite{Nadin} and Nolen \it{et al.} \cite{NRRZ}. Such transition wave for space-heterogenous reaction-diffusion equation connects two steady states, which are constants, contrary to our case that the stationary shock wave connects two steady states, which are non-constant periodic solutions.
}

Our main result is the following:
\begin{theo}(Existence of standing shocks)\label{thm:shock-ex}
Assume that $A\in W^{1, \infty}_{loc} (\TT^N\times\RR)^N$, and that there exist two periodic solutions $v(\cdot, p_+), v(\cdot, p_-)$ to \eqref{cell} with $p_+\neq p_-$, satisfying the following conditions:
\begin{subequations}
\begin{align}
&\bar A_1(p_-)= \bar A_1(p_+)=:\alpha, \label{RH-1}\\
&\bar A_1(p)< \alpha,~ \forall p\in (p_+,p_-) ~\mbox{if}~ p_+<p_-,\quad\bar A_1(p)>\alpha,~ \forall p\in (p_-,p_+) ~\mbox{if}~ p_-<p_+. \label{RH-2}
\end{align}
\end{subequations}
Then there exists a stationary  shock $\bar V$ with  end states $v(\cdot, p_-)$ and $v(\cdot, p_+)$.

\end{theo}



\begin{remark}\label{remark-flux}
The first assumption \eqref{RH-1} is an analogue of the Rankine-Hugoniot condition for standing shock waves of homogeneous conservation laws. The second assumption \eqref{RH-2} is the analogue of the Oleinik condition. It is proved in section \ref{sec:properties} that the Rankine-Hugoniot condition is in fact a necessary condition for the existence of a shock wave.
\end{remark}

Theorem \ref{thm:shock-ex} is proved by passing to the limit in a sequence of approximate problems. In these approximate problems, the domain $\bbr\times \bbt^{N-1}$ is replaced by $(-R,R)\times \bbt^{N-1}$ for some $R>0$. Standard tools of elliptic theory (Harnack inequality, maximum principle, comparison principle, regularity estimates) are used to prove that the approximate sequence enjoys several nice properties, such as monotony and $L^\infty$ bounds.

From now on, we only handle the first case of \eqref{RH-2}, i.e., 
\beq\label{p-assume}
p_+<p_-,\quad \bar A_1(p)<\alpha, \quad \forall p\in (p_+, p_-),
\eeq
the argument for the other case is exactly identical.

\begin{theo} (Stability of standing shocks) \label{thm:stability}
Assume the hypotheses of Theorem \ref{thm:shock-ex}, furthermore $A \in W^{3,\infty}_{loc} (\bbt^N \times \bbr)^N$. Let $\bar{U}$ be a stationary shock wave connecting $v(\cdot, p_-)$ to $v(\cdot, p_+)$, and $u_0\in \bar U +L^1(\bbr\times\bbt^{N-1})$ be a initial perturbation such that
\be\label{hyp:u0}
v(x,p_+)\leq u_0(x)\leq v(x,p_-)\quad \text{for a.e. } x\in \bbr\times\bbt^{N-1},
\ee
and $u=u(t,x)$ be the unique entropy solution of \eqref{main} with $u_{|t=0}=u_0$. 

\begin{itemize}

\item Assume that $\int_{\bbr \times \bbt^{N-1}} (u_0-\bar U)=0$. Then
$$
\lim_{t\to \infty} \|u(t)-\bar U\|_{L^1(\bbr\times\bbt^{N-1})}=0.
$$

\item Assume that $A \in  W^{3, \infty}_{loc} (\TT^N\times\RR)^N$,  that $\int_{\bbr \times \bbt^{N-1}} (u_0-\bar U)\neq 0$ and that there exist functions $\phi, \psi \in L^1(\bbt)$ such that
\begin{align}\label{Lax-1}
\begin{aligned}
&\partial_v A_1(x,v(x,p_-))\geq  \phi(x_1),\quad\mbox{for a.e.}~ x\in\bbt^{N},\\
&a_-:=\int_{\bbt}\phi dx_1 >0,
\end{aligned}
\end{align}
and 
\begin{align}\label{Lax-2}
\begin{aligned}
&\partial_v A_1(x,v(x,p_+))\leq  \psi (x_1),\quad\mbox{for a.e.}~ x\in\bbt^{N},\\
&a_+:=\int_{\bbt}\psi dx_1 <0.
\end{aligned}
\end{align} 

Then there exists a stationary shock $\bar V$ connecting $v(\cdot, p_-)$ to $v(\cdot, p_+)$ such that $u_0-\bar V \in L^1(\bbr \times \bbt^{N-1})$ and
\[
\int_{\bbr\times \bbt^{N-1}} (u_0-\bar V)=0\quad\mbox{and}\quad\lim_{t\to\infty} \|u(t)-\bar V\|_{L^1(\bbr\times\bbt^{N-1})}=0.
\]

\end{itemize}
\end{theo}

\begin{remark}\label{remark-Lax}
$\bullet$ The assumptions \eqref{Lax-1} and \eqref{Lax-2} are the analogue of the Lax conditions for standing shock waves of homogeneous conservation laws. They are used in the present context to obtain a rate of convergence  of stationary shocks towards their end states $v(\cdot,p_\pm)$. This rate of convergence yields some $L^1$ compactness for an approximate problem (see \eqref{eq:approx-p}). We refer to the proofs of Lemma \ref{lem:p} and Proposition \ref{prop:shock-mass} below for details.

$\bullet$ The proof of Theorem \ref{thm:stability} uses classical arguments, relying on tools from dynamical system theory. The main difficulty lies in the second part of Theorem \ref{thm:stability}, which requires, for any real number $q$ and any shock $\bar U$, to find a shock $\bar V$ with the same end states as $\bar U$ and such that $\int (\bar V - \bar U)=q$. This fact is almost obvious in the homogeneous case, since any spatial translate of a shock is a shock. This statement is still rather easy to prove in the 1d case, since a whole family of shocks depending continuously on a parameter is constructed. In the present case, Theorem \ref{thm:shock-ex} only gives the existence of a single shock, and therefore the existence of shocks satisfying the above statement for any $q\in \bbr$ is far from trivial, and is proved in Proposition \ref{prop:shock-mass}.

$\bullet$ Assumption \eqref{hyp:u0} is a classical assumption within the framework of shock stability for conservation laws (see \cite{Serre} and the discussion on initial data within the interval $[u_+, u_-]$ or outside that interval). In order to remove it, we would typically need to prove the stability of the periodic  solutions $v(\cdot, p_\pm)$ under zero-mass perturbation in the space $L^1(\bbr \times \bbt^{N-1})$. However, to our knowledge, the stability of the functions $v(\cdot, p_\pm)$ is known in $L^1(\bbr^N)$ and in $L^1(\bbt^N)$ (see respectively \cite{Dalibard-jems} and \cite{Dalibard-indiana}), but not in $L^1(\bbr \times \bbt^{N-1})$. Furthermore, the proofs of stability in the whole space $\bbr^N$ and in the torus $\bbt^N$ rely on very different arguments, since in the whole space, dispersive effects take place. It is possible that a hybrid proof could be worked out in spaces of the form $\bbr^k\times \bbt^l$ with $k+l=N$, but such a question goes beyond the scope of this paper and thus we choose to leave it open.
\end{remark}

We now provide some examples of fluxes satisfying assumptions \eqref{RH-1}-\eqref{RH-2}, and \eqref{Lax-1}-\eqref{Lax-2}.
 Let $\Phi: \bbt^N\to \bbr^N$ be a divergence-free vector field, $f\in \mathcal C^1(\bbr,\bbr)$, and let $A(x,v):=\Phi(x)f(v)$. Then   for any constant $p\in \bbr$, $v(\cdot,p):=p$ is a solution to the elliptic equation \eqref{cell} with $<v(\cdot, p)>=p$. As a consequence,
 \[
 \bar{A}_1(p)=\int_{\bbt^{N}}A_1(x,p)dx=f(p)\mean{\Phi_1}.
 \]
 Thus \eqref{RH-1} holds if and only if $f(p_+)=f(p_-)$, and \eqref{RH-2} holds if and only if $f(p)-f(p_\pm)$ has the same (strict) sign as $\mean{\Phi_1}(p_+-p_-)$ for $p\in (p_+,p_-)$. For instance, if $\mean{\Phi_1}>0$ and $f(p)=p^2$, any couple $p_-=-p_+>0$ works.\\
 Moreover,
 $$
 \p_v A_1(x,v(x,p))= \Phi_1(x) f'(p),
 $$
 and therefore \eqref{Lax-1}-\eqref{Lax-2} are satisfied for instance if there exists $\alpha>0$ such that $\Phi_1(x) \geq \alpha$ for all $x$, and if $f$ is strongly convex and such that $f(p_+)=f(p_-)$, with $p_+<p_-$.\\

Let us now introduce some notation that will be used throughout the paper. We will often denote the spatial domain by 
$$
\Omega:=\bbr\times \bbt^{N-1}.
$$
In a similar way, we define, for $R>0$,
$$
\Omega_R:= (-R,R)\times \bbt^{N-1}.
$$
We introduce the space $L^1_0(\bbr\times \bbt^{N-1})$ of integrable functions with zero mass
$$
L^1_0(\bbr\times \bbt^{N-1}):=\left\{f\in L^1(\bbr\times \bbt^{N-1}),\ \int_{\bbr\times \bbt^{N-1}} f=0\right\}.
$$
For any integer $k\in \bbz$, and any function $f\in L^1_{loc}(\Omega)$, we define
$$
\tau_k f(x):= f(x+ke_1),\quad \forall x\in \bbr\times \bbt^{N-1}.
$$

{Let us stress that the main difficulty in this article lies in proving the existence of shock waves. Indeed, shock stability follows from classical arguments in \cite{Dalibard-indiana} relying on dynamical system theory (see \cite{OR}). We recall the arguments in section \ref{sec:stability} for the reader's convenience, but the largest part of the paper is devoted to the existence of shocks.}

The paper is organized as follows: section \ref{sec:shock-ex} is devoted to the proof of Theorem \ref{thm:shock-ex}. In section \ref{sec:properties}, we review some properties of stationary shocks. Eventually, section \ref{sec:stability} is devoted to the proof of Theorem \ref{thm:stability}.

\section{Proof of Theorem \ref{thm:shock-ex}} \label{sec:shock-ex}

In this section, we construct stationary shocks thanks to an approximation scheme on compact sets, and then pass to the limit. The proof makes an extensive use of the maximum principle and of the Rankine-Hugoniot \eqref{RH-1} and Oleinik conditions \eqref{RH-2}. 

Before addressing the proof, we first recall some properties of the functions $v(\cdot, p)$ (see \cite{Dalibard-jmpa}):
\begin{prop}\label{prop-add} Assume that the hypotheses of Proposition \ref{prop-exist} are satisfied.
The family $(v(\cdot,p))_{p\in\bbr}$ satisfies the following properties:\\
(i) Regularity estimate : For all $p\in \bbr$, $v(\cdot, p)\in W^{2,q}(\bbt^N)$ for all $1<q<\infty$ and
\[
\forall R>0, ~\exists C_R>0\quad\mbox{s.t.}\quad \sup_{p\in[-R,R]}\|v(p)\|_{W^{2,q}(\bbt^N)}\le C_R.
\]
(ii) Growth property : if $p>p^{\prime}$, then
\[
v(x,p)<v(x,p^{\prime}),\quad x\in \bbt^N
\]

(iii) p-derivative : For all $p\in \bbr$, $\partial_p v(\cdot, p)\in H^{1}(\bbt^N)$ and
\[
\forall R>0, ~\exists C_R>0\quad\mbox{s.t.}\quad \sup_{p\in[-R,R]}\|\partial_p v(p)\|_{H^{1}(\bbt^N)}\le C_R.
\]
Moreover, 
\beq\label{v-inc}
\partial_p v(x,p) >0\quad\mbox{a.e.}~ (x,p)\in \bbt^N\times \bbr.
\eeq
(iv) Behavior at infinity : if additionally $\partial_v A_i \in L^{\infty}(\bbt^N\times \bbr)$ for $1\le i\le N$, and 
\[
\sup_{v\in\bbr} \|\partial_v A(\cdot, v)\|_{L^{\infty}(\bbt^N)} < \infty,
\]
then 
\beq\label{infinite}
\lim_{p\to-\infty} \sup_{x\in \bbt^N} v(x,p) =-\infty,\quad \lim_{p\to+\infty} \inf_{x\in \bbt^N} v(x,p) =+\infty.
\eeq

\label{prop:v-properties}
\end{prop}

\subsection{Construction of approximate solutions}
For any $R>1$, consider the approximate equation:
\beq\label{approximate}\begin{aligned}
- \Delta \bar U_R + \dv A(x, \bar U_R)&=0 \quad \text{in } (-R,R)\times \TT^{N-1},\\
\bar U_R(\pm R, x')&= v(\pm R, x', p_\pm)\quad \forall x'\in \TT^{N-1}.
\end{aligned}
\eeq

For the time being, we assume that the flux $A$ satisfies the assumptions of Proposition \ref{prop-exist} with $m=0$ and $n<1$, i.e. $A$ is uniformly Lipschitz with respect to its second variable, and $\dv_x A$ has sublinear growth. These assumptions will be removed in Remark \ref{rem:growth-A}. 

In this paragraph, we prove the existence and uniqueness of solutions of \eqref{approximate} for any $R>1$. Using the family $v:\bbt^{N}\times\bbr\to \bbr$ constructed in Proposition \ref{prop-exist}, we consider a composite function $V(x):=v(x,f(x_1))$ for some $f\in \mathcal C^\infty(\RR)$ with $f(x_1)=p_-$ if $x_1\leq -1$, $f(x_1)=p_+$ if $x_1\geq 1$. Then we see that \eqref{approximate} is equivalent to
\beq\label{approx-bis}
\begin{aligned}
- \Delta U_R + \dv B(x,  U_R)&=S \quad \text{in } \Omega_R,\\
U_R(\pm R, x')&=0,
\end{aligned}
\eeq
where $U_R:= \bar U_R - V$ and $S:= \Delta V - \dv A(x,V)$, $B(x, r):= A(x, V+r) - A(x, V)$. Notice that since $A$ is  uniformly Lipschitz with respect to $r$, there exists a constant $C$ such that
$$
|B(x, r)|\leq C |r|\quad \forall x\in \Omega,\ \forall r\in \bbr.
$$
Moreover, according to the definition of $S$ and to Proposition \ref{prop-exist}, the support of the function $S$ is included in $[-1,1]\times \bbt^{N-1}$, and $S\in L^2(\Omega)$.

Therefore, it is enough to prove the existence of \eqref{approx-bis}. We want to apply  Schaeffer's fixed point theorem. Let us consider the continuous mapping $L_R: H^1_0(\Omega_R) \to H^1_0(\Omega_R)$ such that $W=L_R(U)$ is the unique solution of the linear elliptic equation:
$$
\begin{aligned}
- \Delta W +\dv B(x,  U) &=S \quad \text{in } \Omega_R,\\
W(\pm R, x')&=0.
\end{aligned}
$$
We use  assumption \eqref{growth-1} with $m=0$ and we obtain
$$
\| \na W \|_{L^2(\Omega_R)}^2 \leq \|S\|_{L^2(\Omega_R)} \|W\|_{L^2(\Omega_R)} + C_0 \|U\|_{L^2(\Omega_R)} \|\na W\|_{L^2(\Omega_R)}.
$$
Using the Poincar$\acute{\mbox{e}}$ inequality and  Young's inequality, we have that
$$
\| \na W \|_{L^2(\Omega_R)}\leq C_R \|S\|_{L^2(\Omega_R)} + C\|U\|_{L^2(\Omega_R)},
$$
for some constant $C_R$ depending on $R$.\\
Since $f$ is smooth, it follows from Proposition \ref{prop-exist} that 
$$
\| \na W \|_{L^2(\Omega_R)}\leq C_R (\|U\|_{L^2(\Omega_R)}+1).
$$
Thus, using the Rellich-Kondrachov theorem, we infer that the mapping $L_R$ is compact. 
Now, there remains to prove that the set
\[
\{ U^\lambda\in H^1_0(\Omega_R) ~|~ U^\lambda=\lambda L_R(U^\lambda),~ \lambda \in [0,1] \}
\]  
is bounded. For any $\lambda \in [0,1]$ and for any solution $U^\lambda$ of $U^\lambda=\lambda L_R(U^\lambda)$, we have
$$
\int_{\Omega_R} | \na U^\lambda|^2 \leq \|S\|_{L^2(\Omega_R)} \|U^\lambda\|_{L^2(\Omega_R)}  + \lambda \left|\int_{\Omega_R} B(x, U^\lambda)\cdot \na U^\lambda\right|.
$$
Let $b:(x,r)\in\RR^{N+1}\mapsto \int_0^r B(x,r')\:dr'$. Then
$$
\int_{\Omega_R} B(x, U^\lambda)\cdot \na U^\lambda=\int_{\Omega_R}\Big(\dv(b(x,U^\lambda)) - (\dv_x b) (x,U^\lambda)\Big) = - \int_{\Omega_R} (\dv_x b) (x,U^\lambda) .
$$ 
Notice that
$$
\dv_x b(x,r)= \int_0^r \Big(\dv A(x,V+r') - \dv A(x,V)\Big)\:dr',
$$
and therefore, using the growth assumption on $A$, there exists a constant $C$ such that for all $r\in \RR$,
$$
|\dv_x b(x,r)|\leq C (1+ |r|^{n+1})\quad \text{with }n<1.
$$
Using once again the Cauchy-Schwartz and the Poincar\'e inequality, we infer that
$$
\| U^\lambda\|_{H^1(\Omega_R)}\leq C_R\quad \forall \lambda\in [0,1].
$$

According to Schaeffer's fixed point theorem, $L_R$ has a fixed point in $H^1_0(\Omega_R)$, and therefore \eqref{approximate} has a solution in $H^1(\Omega_R)$.

Uniqueness follows for instance from the following argument. Let $\bar U_R, \bar U_R'$ be two solutions of \eqref{approximate}, and let $W:= \bar U_R - \bar U_R'$. Then $W$ solves an elliptic equation of the type
$$
\begin{aligned}
-\Delta W + \dv (a_R W)&=0\quad \text{in } \Omega_R,\\
W(\pm R, x')&=0,
\end{aligned}
$$
where $a_R\in L^\infty(\Omega_R)$ is defined by
\[
a_R(x):=\int_0^1\partial_vA(x,\tau \bar U_R(x) + (1-\tau)\bar U_R'(x)) d\tau.
\]  

On the other hand, using the strong form of the Krein-Rutman Theorem (see Appendix), it can be proved that the equation
$$
\begin{aligned}
-\Delta w + \dv (a_R w)=0\quad \text{in } \Omega_R,\\
- \partial_1 w + a_{R,1} w=0 \text{ on } \partial \Omega_R
\end{aligned}
$$
admits a unique positive solution $w\in \mathcal C(\bar \Omega_R)\cap H^1(\Omega_R)$ such that $\int_{\Omega_R }w=1$. A straightforward computation (see \cite{MMP}) shows that
$$
- \Delta \left(\frac{W^2}{w} \right) + \dv \left(a_R\frac{W^2}{w}  \right) = - 2 w\left| \na \frac{W}{w} \right|^2 \quad \text{in } \Omega_R.
$$
Integrating over $\Omega_R$, we deduce that
$$
\int_{\Omega_R} w\left| \na \frac{W}{w} \right|^2=0,
$$
which implies that $W/w$ is constant, therefore $W\equiv 0$ due to $W=0$ at $x_1=\pm R$.

\subsection{Properties of approximate solutions}
We claim that the approximate solution $\bar U_R$ satisfies the following properties.
\begin{lemma}\label{lem:UR}
For any fixed integer $R>1$, let $\bar U_R$ be the  solution of \eqref{approximate}. Then the following properties holds.
\begin{enumerate}
\item A priori bound in $L^\infty$: for all $x\in \Omega_R$,
$$
v(x,p_+)\leq \bar U_R(x)\leq v(x,p_-).
$$
\item Integration constant: there exists a number $\alpha_R$ such that for all $x_1\in (-R,R)$,
\beq\label{int-cst}
-\frac{d}{dx_1} \int_{\TT^{N-1}} \bar U_R(x_1, x')\:dx' +  \int_{\TT^{N-1}} A_1(x_1,x', \bar U_R(x_1, x'))\:dx' = \alpha_R,
\eeq
and $\alpha \leq \alpha_R \leq C$ for some constant $C$ independent of $R$. (Recall $\alpha:=\bar A_1(p_-)= \bar A_1(p_+)$)

\item Monotony: for all $x\in (-R, R-1)\times \TT^{N-1}$,
$$
\bar U_R(x_1+1, x') < \bar U_R(x_1,x').
$$
\item Uniform local a priori bound:  for any $q\in (1,\infty)$, there exists a constant $C_q$ (independent of $R$) such that
$$
\sup_{k\in \{-R, \cdots R-1\}} \|\bar U_R\|_{W^{2,q} ((k, k+1)\times \TT^{N-1})} \leq C_q.
$$

\end{enumerate}
\label{lem:properties}

\end{lemma}
\begin{proof}

For the time being, we still assume that the flux $A$ satisfies the assumptions of Proposition \ref{prop-exist} with $m=0$ and $n<1$, which will be removed in Remark \ref{rem:growth-A}.

\begin{enumerate}
\item A priori bound in $L^\infty$: 

First, notice that using elliptic regularity results together with a bootstrap argument, it is easily proved that $\bar U_R\in W^{2,q}(\Omega_R)$ for all $q<\infty$, and therefore 
 $\bar U_R \in \mathcal C(\overline{\Omega_R})$. Thus, thanks to \eqref{infinite} in Proposition \ref{prop-add} and to the assumption $m=0$, there exist $\bar p_R, \underline p_R$ with $\bar p_R> \underline p_R$ such that
$$
v(x, \underline p_R) \leq \bar U_R (x)\leq v(x, \bar p_R)\quad \forall x\in \bar \Omega_R.
$$

Let us choose $\bar p_R$ (resp. $\underline p_R$) as the smallest (resp. the largest) real number such that the above inequality is satisfied. Then necessarily, since $\bar U_R$ and $v(x,\bar p_R)$ are continuous and $\bar \Omega_R$ is compact, there exists $x_R\in \overline{\Omega_R}$ such that $\bar U_R(x_R)= v(x_R, \bar p_R)$. Let us argue by contradiction, and assume that $x_R$ is an interior point of $\Omega_R$.\\
Notice that $g_R:= v(x,\bar p_R) - \bar U_R$ is a non-negative solution of an elliptic equation of the type
$$
-\Delta g_R + \dv (a g_R)=0\quad \text{in } \Omega_R,
$$
where $a\in L^\infty(\Omega_R)$ is defined by
\[
a(x):=\int_0^1\partial_vA(x,\tau v(x,\bar p_R) + (1-\tau) \bar U_R(x)) d\tau.
\]  
Since $x_R$ is an interior point and $g_R(x_R)=0$, by the Harnack inequality, we have that $g_R$ vanishes on any compactly embedded subset of $\Omega_R$. Thus by continuity, $g_R\equiv 0$ on $\Omega_R$, which is in the contradiction with $\bar U_R \in \mathcal C(\overline{\Omega_R})$ and $p_+\neq p_-$.\\
Therefore, $x_R\in \partial \Omega_R$, thus $\bar p_R\in \{p_+, p_-\}$. Since $p_+<p_-$, we have $\bar p_R=p_-$.\\
Similar arguments lead to $\underline p_R=p_+$.

\item Integration constant:

Integrating  equation \eqref{approximate} on $\TT^{N-1}$ with $x_1$ fixed, we obtain
$$
- \frac{d^2}{dx_1^2}  \int_{\TT^{N-1}} \bar U_R(x_1, x')\:dx' +\frac{d}{dx_1}   \int_{\TT^{N-1}} A_1(x_1,x', \bar U_R(x_1, x'))\:dx' = 0\quad \forall x_1\in (-R,R),
$$
which provides  identity \eqref{int-cst}. Furthermore, notice that since for all $x_1 \in (-R, R-1)$,
$$
\alpha_R= \int_{\TT^{N-1}} \left(\bar U_R(x_1, x') - \bar U_R(x_1+1, x') \right)dx'+ \int_{x_1}^{x_1+1}\int_{\TT^{N-1}} A_1(x, \bar U_R(x))\:dx,
$$
the boundedness of $\bar U_R$ implies that $\alpha_R$ is bounded. Thus, there remains to prove the lower bound $\alpha_R\geq \alpha$. To this end, we consider  identity \eqref{int-cst} at $x_1=-R$ (notice that \eqref{int-cst} holds at $x_1=-R$ because $\bar U_R$ is smooth). Using the boundary condition, we have
$$
\begin{aligned}
\alpha_R&= \frac{d}{dx_1} \int_{\TT^{N-1}} \left( v(x_1, x', p_-) - \bar U_R(x_1, x')\right)\:dx' \Big|_{x_1=-R}
\\&+  \int_{\TT^{N-1}} A_1(-R,x', v(-R,x', p_-))\:dx' -   \frac{d}{dx_1} \int_{\TT^{N-1}}  v(x_1,x', p_-) dx' \Big|_{x_1=-R}.
\end{aligned}
$$
Since 
$v(x_1, x', p_-) - \bar U_R(x_1, x')\geq 0$ for all $(x_1,x')\in \Omega_R$, with equality at $x_1=-R$, we have that
$$
 \frac{d}{dx_1} \int_{\TT^{N-1}} \left( v(x_1, x', p_-) - \bar U_R(x_1, x')\right)\:dx'  \Big|_{x_1=-R}\geq 0.
$$
On the other hand, since
$$
- \Delta v(x,p_-) + \dv A(x, v(x,p_-))=0,
$$
we also have that
$$
 \int_{\TT^{N-1}} A_1(x_1,x', v(x_1, p_-))\:dx' -   \frac{d}{dx_1} \int_{\TT^{N-1}}  v(x_1,x' p_-) dx'=\text{constant}\quad \forall x_1\in \bbr.
$$
Integrating the above identity over $\TT$, we deduce that the above constant is $\bar A_1(p_-)= \alpha$. 
Choosing $x_1=-R$, the inequality $\alpha_R\geq \alpha$ is proved.

\item Monotony:

Consider the function $\bar U_R(x_1+1, x')$ defined on $(-R-1, R-1)\times \TT^{N-1}$. Since the flux $A$ is periodic, $\bar U_R(x_1+1, x')$ satisfies the same equation as $\bar U_R$. Moreover, using the $L^\infty$ a priori estimates and the boundary conditions on $\bar U_R$, we have that
$$
\bar U_R(x_1+1, x') - \bar U_R(x_1, x') \leq 0\quad \text{at } x_1=-R\text { and at }x_1=R-1.
$$
Set $H_R:=\bar U_R(x_1+1, x') - \bar U_R(x_1, x')$ and $(H_R)_-:=-H_R\mathbf 1_{H_R \leq 0 }$. Since the function $x\mapsto -x\mathbf1_{x \leq 0}$ is convex, we have that in $ \mathcal{D}'((-R, R-1) \times \TT^{N-1})$,
\beq\label{HR}
- \Delta (H_R)_- + \dv_x \left( -\mathbf 1_{H_R \leq 0 } (A(x, \bar U_R(x_1+1, x')) -A(x,\bar U_R(x_1, x')) )\right) \leq 0.
\eeq
We denote by $-m_R$ the left hand-side of the above inequality. Then $m_R$ is a non-negative measure. Moreover, straightforward integrations entail
\beq\label{right}
\begin{aligned}
&-m_R((-R, R-1) \times \TT^{N-1})\\
&\quad =\left[ -\partial_1\int_{\TT^{N-1}}  (H_R)_- dx' \right]^{x_1=R-1}_{x_1=-R}\\
&\qquad + \left[ \int_{\TT^{N-1}}  -\mathbf 1_{H_R \leq 0 } \Big(A_1(x,\bar U_R(x_1+1, x')) -A_1(x,\bar U_R(x_1, x')) \Big) dx' \right]^{x_1=R-1}_{x_1=-R}.
\end{aligned}
\eeq
Since $ H_R\leq 0$ at $x_1=-R$ and at $x_1= R-1$, we have 
\[
\partial_1 (H_R)_- = -\mathbf 1_{H_R \leq 0 }\partial_1 H_R= -\partial_1 H_R \quad \text{at }x_1=-R\text{ and } R-1. 
\] 
Using \eqref{int-cst}, we have that
$$
\begin{aligned}
&-\left(\partial_1\int_{\TT^{N-1}}  (H_R)_- dx' \right)\Big|_{x_1=R-1}\\
&\quad + \left( \int_{\TT^{N-1}}  -\mathbf 1_{H_R \leq 0 } \Big(A_1(x,\bar U_R(x_1+1, x')) -A_1(x,\bar U_R(x_1, x')) \Big) dx' \right)\Big|_{x_1=R-1}\\ 
&= \partial_1\int_{\TT^{N-1}} \bar U_R(R, x')\:dx'-\int_{\TT^{N-1}}   (A_1(R,x',\bar U_R(R, x')) dx'\\
&\quad -\partial_1\int_{\TT^{N-1}}\bar  U_R(R-1, x')\:dx'+\int_{\TT^{N-1}}   (A_1(R-1, x',\bar U_R(R-1, x')) dx'\\
&= \alpha_R-\alpha_R=0.
\end{aligned}
$$
Similarly we have the same result at $x_1=-R$. Therefore, it follows from \eqref{right} that $m_R((-R, R-1) \times \TT^{N-1})=0$, and thus
$$
- \Delta (H_R)_- + \dv_x \left( -\mathbf 1_{H_R \leq 0 } (A(x,\bar U_R(x_1+1, x')) -A(x,\bar U_R(x_1, x')) )\right) =0.$$
That is, $(H_R)_-$ is a non-negative solution of an elliptic equation of the type:
\[
- \Delta (H_R)_- + \dv_x \left(a_R(H_R)_- \right) =0,
\]
where $a_R\in L^{\infty}((-R, R-1) \times \TT^{N-1})$ is defined by
\[
a_R:=\int_0^1\partial_vA(x,\tau \bar U_R(x+e_1)+ (1-\tau)\bar U_R(x)) d\tau.
\] 
Let us argue by contradiction and assume that $H_R(x_0)\ge 0$ for some $x_0\in (-R, R-1) \times \TT^{N-1}$. Then $(H_R)_-(x_0)=0$,and Harnack's inequality implies that $(H_R)_-= 0$ on $(-R, R-1) \times \TT^{N-1}$.\\
In that case, $H_R(x)\ge 0$ for all $x\in (-R, R-1) \times \TT^{N-1}$, and since $R$ is an integer, we obtain
\[
v(x,p_+)=\bar U_R(R, x')\ge \bar U_R(R-1, x')\ge \cdots \ge \bar U_R(-R, x')=v(x,p_-).
\]
This is in contradiction with $p_+<p_-$ as \eqref{p-assume}. Therefore, we deduce that $H_R(x)<0$ for all $x\in (-R, R-1) \times \TT^{N-1}$, and
$$
\bar U_R(x_1+1, x') - \bar U_R(x_1, x') < 0\quad \forall x\in (-R, R-1) \times \TT^{N-1}.
$$
\item Uniform bounds in Sobolev spaces:

For any $k\in \{-R+1, \cdots, R-2\}$, $\bar U_R$ solves the equation
$$
-\Delta \bar U_R= - \dv A(x,\bar U_R) \text{ in }(k,k+1)\times\TT^{N-1},
$$
with the inherited boundary conditions, which are bounded in $L^\infty(\TT^{N-1})$ uniformly in $R$ and $k$ thanks to (1). Using the $L^\infty$ a priori bound together with interior elliptic estimates, we infer that $\bar U_R$ is bounded in $H^1((k,k+1)\times\TT^{N-1})$ (and even in $W^{1,q}$ for any $q<\infty$), uniformly in $k$ and $R$. Using a classical bootstrap argument, we then prove that $\bar U_R$ is bounded in $W^{2,q}((k,k+1)\times\TT^{N-1})$ for any $q<\infty$. Using the fact that the boundary conditions at $-R$ and $R$ are smooth and bounded, we derive similar bounds on $(-R,-R+1)\times \TT^{N-1}$ and $(R-1, R )\times \TT^{N-1}$. Hence the result follows.

\end{enumerate}

\end{proof}

\begin{remark}
\label{rem:growth-A}
We here explain how we can remove the constraints $m=0$ and $n<1$ on growth assumptions of the flux. Assume that $A$ belongs to $W^{1,\infty}_{loc}(\TT^N\times \RR)$ and that there exist two periodic solutions $v(\cdot, p_\pm)$ of \eqref{cell} with $p_+\neq p_-$ satisfying \eqref{RH-1}.\\
Let
\[
C_0:=\max (\| v(\cdot, p_+)\|_\infty, \| v(\cdot, p_-)\|_\infty).
\] 
and $\chi\in \mathcal C^\infty_0(\RR)$ such that $\chi(\xi)= 1$ for $|\xi|\leq C_0+1$. Define
$$
A_\chi(x,\xi):= A(x,\xi)\chi(\xi), \quad x\in \TT^N,\ \xi \in \RR
$$
Then the flux $A_\chi$ belongs to $W^{1,\infty}(\TT^N\times \RR)$ and satisfies the growth assumptions of Proposition \ref{prop-exist} with $m=n=0$. Therefore, for any $p\in \RR$ there exists a unique periodic solution $v_\chi(\cdot, p)$ of
$$
- \Delta v_\chi(x) + \dv A_\chi(x, v_\chi(x,p))=0\text{ in } \TT^N,\quad \mean{v_\chi(\cdot, p)}=p.
$$
It follows from the uniqueness of $v_\chi$ and from the definition of $A_\chi$ that $v_\chi(\cdot, p_\pm)= v(\cdot, p_\pm)$.

Now, we can apply the results proved above to the flux $A_\chi$. Thus there exists a unique solution $\bar U^\chi_R$ of equation \eqref{approximate} with $A$ replaced by $A_\chi$, and $\bar U^\chi_R$ enjoys the properties of Lemma \ref{lem:properties}. In particular, $\| \bar U^\chi_R\|_\infty \leq C_0$, and thus
$$
A_\chi(x, \bar U^\chi_R)= A(x, \bar U^\chi_R)\quad \forall x\in \Omega_R.
$$
Hence $\bar U^\chi_R$ is also a solution of \eqref{approximate} with the original flux $A$. Thus we can now drop the $\chi$'s, and consider arbitrary fluxes $A\in W^{1,\infty}_{loc}(\TT^N\times \RR)$ satisfying the assumptions of Theorem \ref{thm:shock-ex}.

\end{remark}

\subsection{Passing to the limit as $R\to \infty$}$ $

$\rhd$\textit{ First step:  Extension to $\RR \times \TT^{N-1}$ and ``normalization''.}

 We first extend $\bar U_R$ to $\RR\times \TT^{N-1}$ by setting 
\beq\label{U_R}
\bar U_R(x)= v(x, p_+)\text{ for } x_1\geq R,\quad \bar U_R(x)= v(x, p_-)\text{ for } x_1\leq- R.
\eeq
Thanks to (4) of Lemma \ref{lem:UR} and to the regularity of $v$, the above function $\bar U_R$ is continuous and bounded uniformly in $R$ in $W^{1,\infty} (\bbr \times \bbt^{N-1})$. Moreover $\bar U_R (\cdot + e_1) \leq \bar U_R$ over the whole space.

Before passing to the limit, one issue is that all integer translations in $x_1$ of shocks are also shocks.  And a shock translated by $ke_1$, with $|k|\gg 1$, is very close to one of the end states $v(\cdot, p_\pm)$ on compact sets in all Sobolev norms. In order to prevent $\bar U_R$ from converging towards $v(\cdot, p_\pm) $, we fix the value (or rather, the mean value) of (a translate of) $\bar U_R$ at a given point. We call this step the ``normalization'' of $\bar U_R$.

More precisely, let $\bar p \in (p_+, p_-)$ be arbitrary (for instance, take $\bar p = \frac{p_++p_-}{2}$). Then since 
\[
\int_0^1\int_{\TT^{N-1}} \bar U_R(R+x_1,x')\:dx'\:dx_1=p_+\quad\mbox{and}\quad 
\int_0^1\int_{\TT^{N-1}} \bar U_R(-R-1+x_1,x')\:dx'\:dx_1=p_-,
\]
there exists $x_R\in (-R-1,R)$ such that 
\[
\int_0^1\int_{\TT^{N-1}} \bar U_R(x_R+x_1,x')\:dx'\:dx_1 = \bar p.
\]
Let $k_R:=\lfloor x_R\rfloor\in \ZZ$ and define $\bar V_R:=\bar U_R(x_1+k_R,x')$. Then since $A$ and $v$ are periodic in their first variable, 
$\bar V_R$ solves
\beq\label{V_R}
\begin{aligned}
&- \Delta \bar V_R + \dv A(x, \bar V_R)=0\text{ in }(-R-k_R, R-k_R)\times \TT^{N-1},\\
&\bar V_R(-R-k_R,x')= v(-R-k_R, x', p_-),\\
& \bar V_R(R-k_R,x')= v(R-k_R, x', p_+),\quad x'\in\bbt^{N-1},
\end{aligned}
\eeq
and there exists $y_R\in [0,1)$ ($y_R= x_R - k_R$) such that
$$
\int_0^1\int_{\TT^{N-1}} \bar V_R(y_R+x_1, x')\:dx'\:dx_1= \bar p.
$$
Additionally, $\bar V_R$ inherits from $\bar U_R$ all the properties listed in Lemma \ref{lem:properties}.
\vskip2mm

$\rhd$\textit{ Second step: Limit $R\to \infty$.}

Thanks to the bounds listed above and in Lemma \ref{lem:properties}, we can extract a subsequence $R_m$ and find a function $\bar V$ such that $\bar V_{R_m} \rightharpoonup \bar V$ in $W^{1,q}(K), ~\forall q\in [1,\infty)$ for any compact set $K\subset \RR\times \TT^{N-1}$, and thus $\bar V_{R_m} \to \bar V$ strongly in $\mathcal C^\alpha(K)$ for some $\alpha>0$. Furthermore, up to a further extraction of a subsequence, there exist some constants $x_+, x_-, \bar y$ and $\bar\alpha$ such that
$$
R-k_R \to x_+ \in [0,+\infty], \quad -R-k_R\to x_-\in [-\infty, 0],\quad y_R\to \bar y \in [0,1],\quad \alpha_R\to \bar \alpha \in [\alpha, C].
$$
Notice also that $x_+ - x_-= +\infty$. Thanks to the strong convergence of $\bar V_{R_m}$ in $\mathcal C^\alpha(K)$, we have
\beq\label{V-p}
\int_0^1\int_{\TT^{N-1}} \bar V(\bar y+x_1, x')\:dx_1\:dx'= \bar p.
\eeq
Furthermore, if $x_+<+\infty$ (resp. $x_->-\infty$), then $\bar V(x_+, x')= v(x_+, x', p_+)$ (resp. $\bar V(x_-, x')= v(x_-, x', p_-)$). 

We can also pass to the limit in \eqref{V_R}, thus $\bar V$ is a solution of
$$
-\Delta \bar V + \dv A(x, \bar V)=0\quad \text{on } (x_-, x_+)\times \TT^{N-1}.
$$

Eventually, we have further properties on $\bar V$ from the properties listed in Lemma \ref{lem:properties} as follows:
\begin{itemize}
\item $L^\infty $ bound : $v(x, p_+) \leq \bar V (x) \leq v(x, p_-)$ for all $x\in\RR\times \TT^{N-1}$;
\item Additional regularity: $\bar V_{R_m}\rightharpoonup \bar V$ in $W^{2,q}(K)$ for any compact set $K \subset (x_-, x_+ )\times \TT^{N-1}$ and for all $q<\infty$, and therefore $\bar V_{R_m}\to \bar V$ strongly in $W^{1,q}(K)$ for such compact sets $K$ and for any $q\in [1,\infty]$. Moreover, $\bar V \in W^{1,\infty} (\Omega)$;
\item Integration constant: 
\beq\label{int-cst-V}
-\frac{d}{dx_1} \int_{\TT^{N-1}} \bar V(x_1, x')\:dx' +  \int_{\TT^{N-1}} A_1(x_1,x', \bar V(x_1, x'))\:dx' = \bar \alpha \quad \forall x_1\in (x_-, x_+).
\eeq
\item Monotony: $\bar V (x + e_1) \leq \bar V(x)$.
\end{itemize}

\vskip2mm

$\rhd$\textit{ Third step: Limit states of $\bar V$ and value of the integration constant.}

Let us consider the sequence $(u_k)_{k\in\bbz}$ defined by
\[
u_k : x\in [0,1]\times \TT^{N-1} \mapsto \bar V  (x + k e_1).
\]
Thanks to the monotony property and the a priori bounds for $\bar V$, the sequence $(u_k)_{k\in\bbz}$ is monotonous and bounded in $W^{1,\infty}( [0,1]\times \TT^{N-1})$. Thus for all $x\in [0,1]\times \TT^{N-1}$, $(u_k(x))_{k\in\bbz}$ has a finite limit as $k\to \pm \infty$, which we denote as $u_\pm$, and $u_\pm$ is bounded and Lipschitz continuous.

Since $u_k(1, x')= u_{k+1}(0, x')$ for all $k\in \ZZ$, we deduce that $u_\pm(0, x')= u_\pm (1,x')$ and thus $u_\pm$ is periodic. Let us now prove that $u_\pm= v(\cdot, p_\pm)$. We consider for instance the function $u_+$, the argument for $u_-$ is strictly identical.

If $x_+<\infty$, since \eqref{U_R} yields
\[
\bar V_R(x_1,x')= v(x_1,x', p_+),\quad x_1\ge R-k_R,~x'\in \bbt^{N-1}
\] 
we deduce easily that $\bar V (x)= v(x, p_+)$ for all $x\in (x_+,\infty)\times\bbt^{N-1}$, and as a consequence, $u_+= v(x, p_+)$.

If $x_+=+\infty$, extending $u_+$ by periodicity, we have $u_+ \in \mathcal C (\TT^N)$ and $\bar V (\cdot + k e_1) \to u_+$ locally uniformly as $k\to + \infty$. Since every $u_k$ is a solution of
$$
-\Delta u_k + \dv A(x, u_k)=0,
$$
taking $k\to \infty$ in the above equation, we deduce that $u_+$ is a periodic solution of the above equation. Therefore there exists $\bar p_+\in \RR$ such that $u_+= v(\cdot, \bar p_+)$. Notice that since $v(\cdot, p_-)\leq \bar V\leq v(\cdot,  p_+)$, we have $\bar p_+ \in [p_+, \bar p_-]$.\\
In particular,
$$
-\frac{d}{dx_1} \int_{\TT^{N-1}} u_+(x_1, x')\:dx' +  \int_{\TT^{N-1}} A_1(x_1,x', u_+(x_1, x'))\:dx' = \bar A_1(\bar p_+)\quad \forall x_1\in\bbt.
$$
Taking the integral of the above identity over $\bbt$ and comparing with \eqref{int-cst-V}, we obtain $\bar A_1(\bar p_+)= \bar \alpha$.\\
Since $\bar\alpha \geq \alpha$ and $\bar p_+ \in [p_+, \bar p_-]$, the assumption \eqref{p-assume} leads to $\bar p_+ \in \{p_+, p_-\}$ and
\beq\label{balpha-alpha}
\bar \alpha= \alpha= \bar A_1(p_\pm).
\eeq
Since $\bar V(\bar y +x_1+ k, x')\leq \bar V(\bar y + x_1, x')$ for all $(x_1, x')\in [0,1]\times \TT^{N-1}$ and $k\in\NN$,
$$
u_+(\bar y + x_1, x')\leq \bar V(\bar y + x_1, x')\quad \forall (x_1, x')\in[0,1]\times \TT^{N-1}.
$$
Taking the average of the above inequality over $[0,1]\times \TT^{N-1}$, it follows from \eqref{V-p} that $\bar p_+ \leq \bar p <p_-$. Therefore $\bar p_+= p_+$.

Hence we conclude that
$$
\bar V(x_1,x')- v(x_1, x', p_\pm) \to 0 \quad \text{as } x_1\to \pm \infty\text{ in } L^\infty(\TT^{N-1}).
$$
Notice also that since \eqref{int-cst-V} is true on $(x_-, x_+)\times \TT^{N-1}$, and at least one of the properties $x_-=-\infty$ or $x_+=+\infty$ always holds, it follows from the argument above that \eqref{balpha-alpha} always holds.

\vskip2mm

$\rhd$\textit{ Fourth step: Conclusion.}

First of all, if $x_+=+\infty$ and $x_-=-\infty$, then gathering the properties of the previous steps, $\bar V $ is a stationary shock with  end states $v(\cdot, p_+)$ and $v(\cdot, p_-)$, thus Theorem \ref{thm:shock-ex} is proved.

Therefore we now consider the case $x_+<+\infty$ (the case $x_->-\infty$ is treated in a similar fashion). In this case, we see that the equation
\beq \label{eq:V}
-\Delta \bar V + \dv A(x,\bar V)=0
\eeq
is satisfied on $(-\infty, x_+)\times \bbt^{N-1}$. Of course, since $\bar V(x)= v(x, p_+)$ for all $x_1>x_+$, the equation is also satisfied on $(x_+, \infty)\times \TT^{N-1}$. Notice that $\bar V$ is continuous at the point $x_1=x_+$, but its derivative in $x_1$ might have a jump, and therefore there might be a Dirac mass in $\Delta \bar V$ at $x_1=x_+$. We prove that this is not the case.\\

Thus, using \eqref{int-cst-V} and \eqref{balpha-alpha}, we have that
\beq\label{f-1}
-\frac{d}{dx_1} \int_{\TT^{N-1}} \bar V(x_1, x')\:dx'\Big|_{x_1=x_+^-} +  \int_{\TT^{N-1}} A_1(x_1,x', \bar V(x_1, x'))\:dx' \Big|_{x_1=x_+^-}=\bar A_1(p_+),
\eeq
and recall that
\beq\label{f-2}
-\frac{d}{dx_1} \int_{\TT^{N-1}} v(x_1, x', p_+)\:dx'\Big|_{x_1=x_+} +  \int_{\TT^{N-1}} A_1(x_1,x',v(x_1, x', p_+))\:dx'\Big|_{x_1=x_+}=\bar A_1(p_+).
\eeq
Let $J(x')$ be the jump of $\p_{x_1} \bar V$ at $x_1=x_+$, i.e.
$$
J(x'):= \partial_{x_1} v(x_+, x', p_+) - \partial_{x_1} \bar V (x_+^-, x').
$$
Since $\bar V(x_+^-, x')= v(x_+, x', p_+)$, combining \eqref{f-1} with \eqref{f-2}, we get 
\[
 \int_{\TT^{N-1}} J(x') dx' =0.
\] 
Moreover, since $\bar V(x)\geq v(x, p_+)$ for all $x\in \RR\times \TT^{N-1}$, with equality for $x_1\geq x_+$, we have $J(x')\geq 0$ for all $x'\in \TT^{N-1}$, consequently $J\equiv 0$. Thus $\partial_{x_1} \bar V$ has no jump at $x_1=x_+$, which implies that the equation \eqref{eq:V} is satisfied over the whole space.\\
Hence $\bar V $ is a stationary shock with  end states $v(\cdot, p_+)$ and $v(\cdot, p_-)$, which completes the proof of Theorem \ref{thm:shock-ex}.

\begin{rem}
In fact, the situation where $x_+<+\infty$ (resp. $x_->-\infty$) cannot happen. Indeed, in that case, $w=\bar V-v(x,p_+)$ (resp. $w=v(x, p_-)- \bar V$) would be the non-negative solution of an elliptic equation of the type
$$
-\Delta w + \dv (aw)=0\quad \text{in } \bbr\times\bbt^{N-1},
$$
with $a\in L^\infty(\bbr\times\bbt^{N-1})$, and $w\equiv 0$ for $x_1\geq x_+$ (resp $x_1\leq x_-$). Using once again the Harnack inequality, we infer that $w$ has to vanish identically over $\bbr\times\bbt^{N-1}$, which leads to a contradiction. Therefore we always have $x_+=+\infty$ and $x_-=-\infty$.
\end{rem}

\section{Properties of stationary shocks with periodic end states}
\label{sec:properties}
We first show that the Rankine-Hugoniot condition \eqref{RH-1} is in fact a necessary condition for the existence of shock waves.

\begin{lemma}\label{lem-indiana}
Assume $A_1\in W^{1, \infty}_{loc} (\TT^N\times\RR)$. Let $\bar{U}$ be a stationary shock wave connecting $v(\cdot, p_-)$ to $v(\cdot, p_+)$.
Then $\bar{A}_1(p_-)=\bar{A}_1(p_+)=:\alpha$, and $\bar{U}$ satisfies
\[
-\frac{d}{dx_1}\int_{\bbt^{N-1}}\bar{U}(x_1,x^{\prime})dx^{\prime} + \int_{\bbt^{N-1}}A_1(x_1,x^{\prime},\bar{U}(x_1,x^{\prime}))dx^{\prime}=\alpha.
\]
\end{lemma}
\begin{proof}
Since the shock wave $\bar U$ is a solution of 
\[
-\Delta \bar U + \dv A(x,\bar U)=0,
\]
there exists a constant $\bar C$ such that
\[
-\frac{d}{dx_1}\int_{\bbt^{N-1}}\bar{U}(x)dx^{\prime} + \int_{\bbt^{N-1}}A_1(x,\bar{U}(x)) dx^{\prime}=\bar C,\quad \forall x_1\in \RR.
\] 
Moreover, since
\[
-\frac{d}{dx_1}\int_{\bbt^{N-1}} v(x,p_{+})dx^{\prime} + \int_{\bbt^{N-1}}A_1(x,v(x,p_{+})) dx^{\prime}=\bar{A}_1(p_+),
\] 
we have
\be\label{int-cst-U}
-\frac{d}{dx_1}\int_{\bbt^{N-1}}\Big(\bar{U}(x)- v(x,p_{+})\Big) dx^{\prime} + \int_{\bbt^{N-1}}\Big(A_1(x,\bar{U}(x)) -A_1(x,v(x,p_+))\Big)dx^{\prime}=\bar C-\bar{A}_1(p_+).
\ee
Notice that \eqref{shock-b} and $A_1\in W^{1, \infty}_{loc} (\TT^N\times\RR)$ yield that for any $\eps\in (0,1)$, there exists $m>0$ such that for all $x_1>m$, 
\[
|\bar{U}(x)- v(x,p_{+})|\le \eps,\quad |A_1(x,\bar{U}(x)) -A_1(x,v(x,p_{+}))|\le \eps\|\partial_v A_1\|_{L^{\infty}(\bbt^N\times (-K,K))} ,
\] 
where the constant $K$ is such that $K\geq \| v(\cdot, p_+)\|_\infty +1$.
Thus, integrating \eqref{int-cst-U} over $[m,m+1]$, we have that
\[
|\bar C-\bar{A}_1(p_+)|\le C\eps\quad \forall \eps\in (0,1),
\] 
which implies that $\bar C=\bar{A}_1(p_+)$. Similarly, applying the above argument to $ v(\cdot,p_{-})$, we have $\bar C=\bar{A}_1(p_-)$.
\end{proof}

If we impose additional conditions on the flux $A$ at the  two end states, the shock wave exponentially converges towards the end states:
\begin{proposition}\label{prop-exp}
Let $\bar{U}$ be a stationary shock wave connecting $v(\cdot, p_-)$ to $v(\cdot, p_+)$ satisfying $v(\cdot, p_+)\le\bar U\le v(\cdot, p_-)$. Assume that $A_1\in (W^{1, \infty}_{loc}\cap  \mathcal{C}^1) (\TT^N\times\RR)$ and that there exist periodic functions $\phi \in L^1(\bbt)$ and $\psi \in L^1(\bbt)$ such that the Lax conditions \eqref{Lax-1}, \eqref{Lax-2} are satisfied.

Then there exist positive constants $R$ and $C_R$ such that for all $\pm x_1>R$, 
\[
\int_{\bbt^{N-1}}|\bar{U}(x_1,x^{\prime})-v(x_1,x^{\prime},p_{\pm})|dx^{\prime}<C_R\, e^{a_{\pm}x_1/2}\int_{\bbt^{N-1}}|\bar{U}(\pm R,x^{\prime})-v(\pm R,x^{\prime},p_{\pm})|dx^{\prime}. 
\]
\end{proposition}
\begin{proof}
We show the convergence towards the left end state $v(\cdot, p_-)$. First of all, we see that Lemma \ref{lem-indiana} yields
\begin{align*}
\begin{aligned}
&\frac{d}{dx_1} \int_{\bbt^{N-1}}|\bar{U}(x)-v(x,p_{-})|dx^{\prime}\\
&\quad=\frac{d}{dx_1} \int_{\bbt^{N-1}}(v(x,p_{-})-\bar{U}(x))dx^{\prime}\\
&\quad= \int_{\bbt^{N-1}}(A_1(x,v(x,p_{-}))-A_1(x,\bar{U}(x)))dx^{\prime}\\
 &\quad=\int_{\bbt^{N-1}}\left(\int_0^1\partial_v A_1(x,\tau \bar{U}+(1-\tau)  v(x,p_-) )d\tau\right)\; (v(x,p_-)-\bar{U}(x)) dx^{\prime}.
\end{aligned}
\end{align*}
It follows from \eqref{shock-b} that 
\[
\forall \eps>0, ~\exists R>0\quad\mbox{s.t.}\quad  |\bar{U}(x)-v(x,p_-)|<\eps\quad\mbox{for all}~x_1 < -R, ~ x^{\prime}\in\bbt^{N-1}.
\]
Since the function $\p_v A_1$ is continuous, using \eqref{Lax-1}, we infer that there exists a constant $C$ such that 
$$
\partial_v A_1(x,\tau \bar{U}+(1-\tau)  v(x,p_-) )\geq \phi(x_1)-C\eps\quad\mbox{for all}~ x_1<-R,~ x^{\prime}\in\bbt^{N-1},\ \tau\in [0,1].
$$
Thus for all $x_1 < -R$,
\[
\frac{d}{dx_1}\int_{\bbt^{N-1}}|\bar{U}(x)-v(x,p_-)|dx^{\prime} \geq  \left(\phi(x_1)-C\eps\right)\int_{\bbt^{N-1}} |\bar{U}(x)-v(x,p_-)| dx^{\prime}.
\]
This inequality implies that for all $x_1 < -R$
\[
\int_{\bbt^{N-1}}|\bar{U}(x)-v(x,p_-)|dx^{\prime}<\int_{\bbt^{N-1}}|\bar{U}(x)-v(x,p_-)|dx^{\prime}\Big|_{x_1=-R}\times \exp{\Big(\int_{-R}^{x_1} (\phi(s)-C\eps) ds\Big)}. 
\]
Since the periodicity of $\phi$ implies that there exists a positive constant $C$ such that 
\begin{align*}
\begin{aligned}
\int_{-R}^{x_1} \phi(s)ds &\le (x_1+R) \int_{\bbt} \phi(s)ds +C,
\end{aligned}
\end{align*}
we have the desired estimate for the case of $v(x,p_-)$, choosing $\eps $ so that $a_--C\eps \geq a_-/2$. The same arguments also lead to the convergence towards $v(\cdot,p_+)$ as $x_1\to +\infty$.
\end{proof}

\begin{lemma}\label{lem-L1}
Assume $A\in W^{1, \infty}_{loc} (\TT^N\times\RR)^N$ satisfying the assumptions of Theorem \ref{thm:shock-ex}, together with \eqref{p-assume}. Let $\bar{U}$ and $\bar{V}$ be two stationary shock waves connecting $v(\cdot, p_-)$ to $v(\cdot, p_+)$. Then $\bar U$ and $\bar V$ enjoy the following properties:
\begin{itemize}
\item The function $\bar U - \bar V$ keeps a constant sign;
\item $\bar U -\tau_1 \bar U \geq 0$, $\bar V - \tau_1 \bar V \geq 0$;
\item There exist $k_-, k_+\in \bbz$, with either $k_+=0$ or $k_-=0$, such that
$$
\tau_{k_-} \bar U \leq \bar V \leq \tau_{k_+} \bar U;
$$
\item $\bar{U}-\bar{V}\in L^1(\bbr\times\bbt^{N-1})$.
\end{itemize}

\end{lemma}
\begin{remark}\label{rem-su}
The first statement of Lemma \ref{lem-L1} implies that stationary shocks are ordered in the sense that any two shocks $\bar{U}$ and $\bar{V}$ satisfy only one of $\bar{U}=\bar{V}$, $\bar{U}<\bar{V}$ and $\bar{U}>\bar{V}$.
\end{remark}
\begin{proof}
 Let us first prove that  $\bar U - \bar V$ keeps a constant sign. Assume for instance that $\bar{U}(0)\le\bar{V}(0)$.  Using the same argument as the one developed from \eqref{HR}, we have that $W:=\bar{U}-\bar{V}$ satisfies
\[
- \Delta |W| + \dv_x \Big( \sgn(W) \left(A(x,\bar{U}) -A(x,\bar{V}) \right)\Big) \leq 0\quad\mbox{in}~\mathcal{D}'(\bbr\times\bbt^{N-1}).
\]
We denote by $-m$ the left hand-side of the above inequality. Then $m$ is a non-negative measure. But since 
\[
\lim_{x_1\to\pm\infty} \| \bar{U} (x_1,\cdot) - \bar{V} (x_1,\cdot)\|_{L^{\infty}(\bbt^{N-1})}=0,
\]
we have that $m(\bbr\times\bbt^{N-1})=0$. Thus,
\[
- \Delta |W| + \dv_x \Big( \sgn(W) \left(A(x,\bar{U}) -A(x,\bar{V}) \right)\Big) =0.
\]
That is, $|W|$ is a non-negative solution of an elliptic equation of the type:
\[
- \Delta |W| + \dv_x \left(a|W| \right) =0,
\]
where
\[
a:=\int_0^1\partial_vA(x,\tau \bar{U}+ (1-\tau)\bar{V}) d\tau \in L^{\infty}(\bbr\times\bbt^{N-1}).
\] 
As a consequence,  Harnack's inequality implies that either $W$ is identically zero, or $W$ never vanishes. Thus there are two possibilities:
\begin{itemize}
\item If $\bar U(0)=\bar V(0)$, then $W\equiv 0$ and $\bar U=\bar V$;
\item If $\bar U(0)<\bar V(0)$, then $W$ never vanishes and $\bar V-\bar U$ remains strictly positive. In that case
\[
v(x,p_+)<\bar{U}(x)<\bar{V}(x)<v(x,p_-),\quad \forall x\in \bbr\times\bbt^{N-1}.
\]
\end{itemize}
Hence the first statement of the Lemma is proved.

Concerning the second statement, observe that $\tau_1 \bar U$ and $\tau_1 \bar V$ are also stationary shock waves connecting $v(\cdot, p_-)$ to $v(\cdot, p_+)$. As a consequence, according to the first statement, $\bar U- \tau_1 \bar U$ and $\bar V -\tau_1 \bar V$ keep a constant sign. It follows that the sequences of functions $(\tau_k \bar U)_{k\in \bbz}, (\tau_k \bar V)_{k\in \bbz}$ are monotonous, and using assumption \eqref{p-assume}, we infer that these sequences are necessarily non-increasing. Hence $\bar U- \tau_1 \bar U\geq 0$, $\bar V -\tau_1 \bar V\geq 0$.

We now address the third statement. Once again, without loss of generality, we assume that $\bar U\leq \bar V$, so that $k_-=0$. Moreover, since the sequence $(\bar U(ke_1))_{k\in\bbz}$ is monotonous, we have
\[
\tau_k \bar U (x) \to v(x,p_-)\quad \mbox{as}~ k\to -\infty \text{ in } L^\infty([0,1]\times \bbt^{N-1})
\]
Thus there exists $k_+\in\bbz$, $k_+\leq 0$, such that for all $x\in [0,1]\times \bbt^{N-1}$,
\[
\bar{V}(x)\leq\tau_{k_+}\bar U(x).
\]
Using the first statement and the fact that $\tau_{k_+}\bar U$ is a standing shock, we infer that $\bar V \leq \tau_{k_+}\bar U$.

Eventually, still working under the assumption $\bar U\leq \bar V$, we have, for any $K>0$,
\begin{eqnarray*}
\int_{-K}^K\int_{\bbt^{N-1}} |\bar U-\bar V|&\leq& \int_{-K}^K\int_{\bbt^{N-1}}(\tau_{k_+}\bar U-\bar U)\\
&=&\int_{-K+k_+}^{-K}\int_{\bbt^{N-1}}\bar U -\int_{K+k_+}^K\int_{\bbt^{N-1}} \bar U\leq 2 |k_+| \|\bar U\|_\infty,
\end{eqnarray*}
and therefore $\bar U-\bar V \in L^1(\Omega)$.
\end{proof}

\begin{proposition}\label{prop:shock-mass}Assume that the assumptions of Theorem \ref{thm:shock-ex} are satisfied, together with the Lax assumptions  \eqref{Lax-1}-\eqref{Lax-2}. Assume furthermore that $\p_v A, \p_v^2 A \in W^{1,\infty}_{loc} (\bbt^N \times \bbr)$.

Let $\bar{U}$ be a stationary shock wave connecting $v(\cdot, p_-)$ to $v(\cdot, p_+)$.
Let $q\in \RR$ be arbitrary. Then there exists a unique shock $\bar V \in \bar U + L^1( \bbr\times\bbt^{N-1})$ such that $\int_{\bbr\times \bbt^{N-1}} (\bar V - \bar U)=q$.
\end{proposition}

\begin{remark}
The sole purpose of  assumptions \eqref{Lax-1}-\eqref{Lax-2} is to ensure that the family $(p_R)_{R>0}$ defined by \eqref{eq:approx-p} below is equi-integrable, and therefore compact with respect to $R$. If this compactness property can be retrieved in another way, then assumptions \eqref{Lax-1}-\eqref{Lax-2} can be removed from the statement of  Proposition \ref{prop:shock-mass}. 
\end{remark}

Proposition \ref{prop:shock-mass} has the following immediate consequence:
\begin{cor}\label{cor-zero}
Let $p_-,p_+$ be constants such that the assumptions of Theorem \ref{thm:shock-ex} are satisfied. Assume that  \eqref{Lax-1}-\eqref{Lax-2} hold,  and let $\bar{U}$ be a stationary shock wave connecting $v(\cdot, p_-)$ to $v(\cdot, p_+)$. If $u\in \bar{U} + L^1( \bbr\times\bbt^{N-1})$, then there exists a unique standing shock $\bar V$ such that $u\in \bar V + L^1_0( \bbr\times\bbt^{N-1})$.
\end{cor}

We now turn to the proof of Proposition \ref{prop:shock-mass}. The proof relies heavily on properties of the function $\bar U(\cdot + e_1)-\bar U$, which we list in the following Lemma:
\begin{lem}\label{lem:p}
Assume the hypotheses of Theorem \ref{thm:shock-ex}, together with \eqref{p-assume}, futhermore $\p_v A \in W^{1,\infty}_{loc}(\bbt^N\times \bbr)$. Let $\bar{U}$ be a stationary shock wave connecting $v(\cdot, p_-)$ to $v(\cdot, p_+)$, and let
$$
p:= \bar U - \bar U(\cdot + e_1).
$$
Then $p$ satisfies the following properties:
\begin{itemize}
\item  Setting 
$$
b(x)=\int_0^1 \partial_v A(x, s\bar U(x) + (1-s)\bar U(x+e_1))\:ds\in{W^{1,\infty}(\Omega)},
$$
the function $p$ is a non-negative solution of
\be\label{eq:p}
-\Delta p + \dv (bp)=0\quad \text{in } \bbr\times\bbt^{N-1};
\ee
Moreover, $p\in L^1\cap W^{1,\infty}(\bbr\times \bbt^{N-1})$.

\item For any $R>1$, consider the approximate problem
\be\label{eq:approx-p}\ba
- \Delta p_R + \dv(b p_R)=0\quad \text{in } (-R,R)\times \bbt^{N-1},\\
-\p_1 p_R + b_1 p_R=0 \text{ at } x_1=\pm R,\\
\int_{\Omega_R} p_R=\int_{\Omega} p.
\ea
\ee
Then equation \eqref{eq:approx-p} has a unique solution  $p_R\in H^1(\Omega_R)$. Moreover, $p_R(x)>0$ for all $x\in \Omega_R$, and if $p_R$ is extended by zero outside $\Omega_R$, the family $(p_R)_{R>0}$ is uniformly bounded in $L^q(\Omega)$ for all $1\leq q<\infty$.

\item Assume that the Lax conditions \eqref{Lax-1}-\eqref{Lax-2} are satisfied. Then
$$
p_R\to p\quad \text{as } R\to \infty, \text{ in } L^1(\bbr\times \bbt^{N-1}).
$$

\end{itemize}

\end{lem}
\begin{proof}
$\bullet$ \textit{Properties of $p$: } the integrability, sign  and regularity properties of $p$ follow from Lemma \ref{lem-L1} and from the regularity properties of $\bar U$. The equation on $p$ simply follows from making the difference between the equations on $\tau_1 \bar U$ and $\bar U$.

$\bullet$ \textit{Properties of $p_R$:} existence, uniqueness and positivity are a consequence of the Krein-Rutman theorem (see Appendix). 
The uniform $L^1$ bound follows from the normalization and  the positivity. We then obtain $H^1$ bounds by multiplying \eqref{eq:approx-p} by $p_R$ and integrating by parts. We obtain
\begin{eqnarray*}
&&\int_{\Omega_R} | \na p_R|^2 \leq\frac{1}{2} \|b\|_{W^{1,\infty}} \left(\int_{\bbt^{N-1}}p_R^2(R, x')\:dx' + \int_{\bbt^{N-1}}p_R^2(-R, x')\:dx' +\int_{\Omega_R} p_R^2 \right).
\end{eqnarray*}
Using first a trace inequality and then the Gagliardo-Nirenberg interpolation inequality, we infer that for all $\nu>0$, there exists a constant $C_\nu$, independent of $R$, and such that
\begin{eqnarray*}
\|p_R(\pm R, \cdot)\|_{L^2(\bbt^{N-1})}&\leq &C \|p_R\|_{H^{1/2}(\Omega_R)}\leq C_\nu \|p_R\|_{L^2} + \nu \|\na p_R\|_{L^2}\\
&\leq & C_\nu \|p_R\|_{L^1}^{1-\alpha}\|\na p_R\|_{L^2}^{\alpha} + \nu  \|\na p_R\|_{L^2}\\
&\leq & C_\nu \|p_R\|_{L^1} + 2 \nu \|\na p_R\|_{L^2},
\end{eqnarray*}
where $\alpha=N/(N+2)$. Taking $\nu$ sufficiently small, we infer that
$$
\int_{\Omega_R}| \na p_R|^2 \leq C  \|b\|_{W^{1,\infty}} \|p_R\|_{L^1}^2 \leq C.
$$
Using once again the uniform $L^1$ bound together with the Gagliardo-Nirenberg interpolation inequality, we obtain
$$
\sup_{R>0}\| p_R\|_{H^1(\Omega_R)}<\infty.
$$


Likewise, we have  uniform $L^q$ bounds, i.e., for any $1<q<\infty$, 
\[
\|p_R\|_{L^q(\Omega_R)} \le \|p_R\|_{L^1}^{1-\alpha'}\|\na p_R\|_{L^2}^{\alpha'}\le C, \quad \alpha'=\frac{2N(q-1)}{(N+2)q}\in (0,1).
\]
 
 $\bullet$ \textit{Asymptotic behaviour of $p_R$ when the Lax conditions are satisfied:}
we first obtain estimates on the rate  of decay in $x_1$ in the following way. Integrating equation \eqref{eq:approx-p} on $\bbt^{N-1}$ leads to
 $$
 - \frac{d^2}{dx_1^2}\int_{\bbt^{N-1}} p_R + \frac{d}{dx_1}\int_{\bbt^{N-1}} b_1 p_R=0,
 $$
 and thus
 $$
 - \frac{d}{dx_1}\int_{\bbt^{N-1}} p_R + \int_{\bbt^{N-1}} b_1 p_R=\text{cst. on } [-R,R].
 $$
 The boundary conditions imply that the constant has to be zero, and therefore
 $$- \frac{d}{dx_1}\int_{\bbt^{N-1}} p_R + \int_{\bbt^{N-1}} b_1 p_R=0\text{  on } [-R,R].$$
 Now, since Proposition \ref{prop-exp} yields that for all $\pm x_1>R$,
 \begin{eqnarray*}
\int_{\bbt^{N-1}} |p| \le \int_{\bbt^{N-1}} |\bar{U}(x)- v(x,p_{\pm})|+\int_{\bbt^{N-1}}|v(x+e_1,p_{\pm})-\bar U(x+e_1)| < C_R\, e^{a_{\pm}x_1/2},
\end{eqnarray*}
there exists $x'_0\in\bbt^{N-1}$ such that $p(x_1,x'_0)\le C_R\, e^{a_{\pm}x_1/2}$ for all $\pm x_1>R$, which implies together with Harnack inequality that $p(x_1, \cdot)$ converges exponentially fast towards zero in $L^\infty(\bbt^{N-1})$ as $x_1\to \pm\infty$, and $b_1$ therefore converges exponentially fast towards
 $
 \p_v A(x, v(x,p_\pm))
 $ in $L^\infty(\bbt^{N-1})$ as $x_1\to \pm\infty$. The Lax conditions \eqref{Lax-1}-\eqref{Lax-2} imply that for any $\eps<\max(|a_+|, |a_-|)/2$, there exists $K>0$ such that
 $$
b_1(x)\geq\phi(x_1)-\eps\quad\text{for } x_1<-K,\quad b_1(x)\leq \psi(x_1) + \eps\quad\text{for } x_1>K.
 $$
 Thus, if $K<x_1<R$, we obtain, since $p_R>0$,
 $$
 0 \leq - \frac{d}{dx_1}\int_{\bbt^{N-1}} p_R + (\psi(x_1) + \eps)\int_{\bbt^{N-1}} p_R .
 $$
 As a consequence, there exists a constant $C$ such that
 $$
 \int_{\bbt^{N-1}} p_R(x_1, x')\:dx'\leq C \exp(x_1 a_+/2).
 $$
We also obtain similar estimates on $(-R,-K)$. 
 Using the Harnack inequality, we deduce eventually that there exists a constant $C$ (independent of $R$) such that
 $$
 C^{-1}\exp(a_- x_1/2)\leq p_R(x)\quad \text{if } x_1<0,\quad p_R(x)\leq C\exp(a_+ x_1/2)\quad \text{if } x_1>0.
 $$
Furthermore, if we consider
\[
\varphi_+:= \inf_{x'\in\bbt^{N-1}}\partial_v A_1(x_1,x',v(x_1,x',p_+)),\quad  \varphi_-:= \sup_{x'\in\bbt^{N-1}}\partial_v A_1(x_1,x',v(x_1,x',p_-)) ,
\]
then \eqref{Lax-1}-\eqref{Lax-2} imply that $b_+:= \int_{\bbt} \varphi_+ dx_1<0$ and $b_-:= \int_{\bbt} \varphi_- dx_1>0$, and using the above arguments, we obtain that 
 $$
 p_R(x)\le C\exp(b_- x_1/2) \quad \text{if } x_1<0,\quad C^{-1}\exp(b_+ x_1/2)\le p_R(x)\quad \text{if } x_1>0.
 $$
Hence, the sequence $(p_R)_{R>1}$ is equi-integrable. Using the uniform $H^1$ estimate, we deduce that $(p_R)_{R>1}$ is compact in $L^1(\Omega)$. By uniqueness (up to a multiplicative constant) of the solutions of \eqref{eq:p}, it follows that $p_R\to p$ in $L^1(\Omega)$.

\end{proof}

\begin{remark}
Obviously, the same statements hold for $p_k:=|\bar U(\cdot +k e_1) - \bar U|$ for any $k\in \bbz$, replacing every occurrence of $\bar U(\cdot + e_1)$ by $\bar U(\cdot +k e_1)$.
\end{remark}

We are now ready to prove Proposition \ref{prop:shock-mass}:
\subsection{Proof of Proposition \ref{prop:shock-mass}}
Let $q\in \bbr$ be fixed. Notice first that if there exist two shocks $\bar V_1, \bar V_2$ with $\int(\bar V_1-\bar U)=\int(\bar V_2-\bar U)=q$, then $\bar V_1-\bar V_2\in L^1_0(\Omega)$ and $\bar V_1-\bar V_2$ keeps a constant sign according to Lemma \ref{lem-L1}. Hence $\bar V_1=\bar V_2$. The uniqueness of $\bar V$ follows. We therefore focus on the existence of $\bar V$ in the rest of the proof.

First, there exists an integer $k\in \bbz$ such that $q$ has the same sign as $\bar U(\cdot + k e_1)- \bar U$, and 
$$
|q|\leq \| \bar U(\cdot + k e_1)- \bar U\|_{L^1}= \|p_k\|_{L^1}.
$$
In order to fix ideas, we work with $q>0$, so that $k<0$ and $p_k= \bar U(\cdot + k e_1)- \bar U$. In the sequel, we set
$$
b_k(x)=\int_0^1 \partial_v A(x, (1-s)\bar U(x) + s\bar U(x+ke_1))\:ds\in W^{1,\infty}(\Omega).
$$

The goal is to prove that for all  $q\in \bbr$, the following equation has at least one solution
\begin{equation}\label{eq:W}
-\Delta W + \dv B(x,W)=0\quad \text{in }\Omega, \quad W\in L^1(\Omega),\  \int_{\Omega} W=q,
\end{equation}
where $B(x,r)=  A(x, \bar U + r) -  A(x, \bar U)$. Setting $W=\bar V - \bar U$, this is strictly equivalent to the statement of Proposition \ref{prop:shock-mass}.

 In order to require that $\int_{\Omega} W =q$, we slightly modify the form of equation \eqref{eq:W} and rather look for solutions of the equation
\begin{equation}
\label{eq:W-2}
-\Delta W + \dv(b_kW) + \dv \tilde B_k(x, W)=0,\quad W \in  L^1(\Omega),\ \int_{\Omega} W=q
\end{equation}
where $\tilde B_k(x,r)= \tilde A(x, \bar U + r) - \tilde A(x, \bar U) - b_k (x)r$.
 Here, $\tilde A$ is defined by $\tilde A(x,r):=A(x,r)\chi(r)$, where $\chi\in \mathcal C^\infty_0(\RR)$ such that $\chi(r)= 1$ for $|r|\leq r_0$ 
for some large constant $r_0$ with $r_0>2\| v(\cdot, p_\pm)\|_\infty$, thus $\tilde A\in W^{1,\infty}(\TT^{N}\times\bbr)$. It is clear that if $W\in L^1$ is a solution of \eqref{eq:W-2}, then $\bar V = W + \bar U$ is a standing shock for the flux $\tilde A $, with periodic end states $v(\cdot,p_\pm)$. As a consequence, $v(\cdot, p_+)\leq\bar V\leq v(\cdot, p_-)$, and thus $\tilde A(x, \bar V(x))= A(x, \bar V(x))$. Whence $\bar V$ is a standing shock for the flux $A$ such that $\int (\bar V - \bar U)= q$. 

Notice also that there exists a constant $C$ such that
$$
|\tilde B_k(x,r)| \leq C |r| \quad \forall r\in \bbr,\ \forall x\in \Omega,
$$
and that for all  $r\in \bbr, x\in \Omega$, since $\p_v^2 A \in L^\infty_{loc}$,
\begin{eqnarray*}
\tilde B_k(x,r)&=&r\int_0^1 \left( \p_v \tilde A(x, \bar U + s r) - \p_v  \tilde A(x, \bar U + s p_k(x))\right)\:ds\\
&=&r (r-p_k) \int_0^1\int_0^1s\p_v^2 \tilde A (x , \bar U + s\tau r + s(1-\tau ) p_k(x))\:d\tau\:ds.
\end{eqnarray*}
As a consequence,  for all  $r\in \bbr, x\in \Omega$,
\beq\label{bound-1}
|\tilde B_k(x,r)| \leq C |r| |r-p_k(x)|.
\eeq

We prove the existence of solutions of \eqref{eq:W-2} by using approximate problems on $\Omega_R$ and passing to the limit as $R\to \infty$. Using Lemma \ref{lem:p}, we first introduce the function $p_{k,R}$ which solves
\begin{align}
\begin{aligned}\label{p_kR}
&-\Delta p_{k,R} + \dv (b_k p_{k,R})=0\text{ in } \Omega_R,\\  
&- \p_1 p_{k,R} + b_{k, 1} p_{k,R}=0 \text{ for } x_1=\pm R,
& \int_{\Omega_{R}} p_{k,R}=\int_\Omega p_k.
\end{aligned}
\end{align}
We recall that $p_{k,R}>0$ in $\Omega_R$. We define
$$
 \tilde B_{k,R}(x,r)=\chi_R(x_1) r (r-p_{k,R}) \int_0^1\int_0^1s\p_v^2 \tilde A (x , \bar U + s\tau r + s(1-\tau ) p_k(x))\:d\tau\:ds
$$
for some cut-off function $\chi_R$ such that $\chi_R\equiv 1$ on $(-R+1, R-1)$ and $\mathrm{Supp}\; \chi_R\subset (-R+1/2, R-1/2)$.

We now prove that for all $R>1$, there exists a solution $W_R\in H^1(\Omega_R)$ of the equation
\begin{equation}
\label{eq:W-approx}\ba
-\Delta W_R + \dv(b_kW_R) + \dv \tilde B_{k,R}(x, W_R)=0\text{ in } \Omega_R,\\
- \p_1 W_R + b_{k, 1} W_R=0 \text{ for } x_1=\pm R,\\
\int_{\Omega_R} W_R=q.\ea
\end{equation}

Let us  solve  equation \eqref{eq:W-approx} by using Schaefer's fixed point theorem. Let $W_1\in H^1(\Omega_R)$ be arbitrary. We  use the Fredholm alternative to solve the equation
\beq\label{Fred-eq}
\ba
-\Delta W_2 + \dv(b_kW_2) + \dv \tilde B_{k,R}(x, W_1)=0\text{ in } \Omega_R,\\
- \p_1 W_2 + b_{k, 1} W_2=0 \text{ for } x_1=\pm R,\\
\int_{\Omega_R} W_2=q.\ea
\eeq
Indeed, according to Lemma \ref{lem-app}, the solutions of the homogeneous equation 
\beq\label{homo-1}
-\Delta w + \dv (b_k w)=0\text{ in } \Omega_R,  \quad - \p_1 w + b_{k, 1} w=0 \text{ for } x_1=\pm R
\eeq
are the functions $w=c p_{k,R}$ where $p_{k,R}>0, c\in\bbr$. Since the dual problem of \eqref{homo-1} is
\[
- \Delta q - b_k\cdot \na q=0\quad \text{in }  \Omega_R,\quad \p_1 q=0\text{ for } x_1=\pm R,
\]
and a simple computation gives
\[
\int_{\Omega_R} p_{k,R} |\nabla q|^2 dx =0,
\]
the solutions of the dual problem are the constants. Thus, the inhomogeneous term $-\dv \tilde B_{k,R}(x, W_1)$ of \eqref{Fred-eq} is orthogonal to the constants thanks to the cut-off function $\chi_R$. This ensures the existence of solutions of the first two lines of \eqref{Fred-eq}; these solutions are defined up to a function of the form $cp_{k,R}$, for $c\in \bbr$, and the third line of \eqref{Fred-eq} fixes the value $c$ and ensures uniqueness of solutions of \eqref{Fred-eq}.
Hence we can define the operator
$L_R:W_1\in L^2(\Omega_R)\mapsto W_2\in L^2(\Omega_R)$. Notice that in fact, the operator $L_R$ is continuous from $L^2(\Omega_R)$ to $H^1(\Omega_R)$, and therefore $L_R$ is compact for all $R>0$. Now, let $\lambda\in [0,1]$ be arbitrary, and let $W^\lambda$ be such that $ \lambda L_R(W^\lambda)=W^\lambda$. We first observe that  since $\tilde B_{k,R}(x,0)=0$, $(W^\lambda)_+:=W^\lambda\mathbf 1_{W^\lambda \geq 0 }$ satisfies
$$\ba
-\Delta (W^\lambda)_+ + \dv( b_k (W^\lambda)_+) + \lambda \dv \left(\mathbf 1_{W^\lambda>0}\tilde B_{k,R}(x, W^\lambda)\right)\leq 0\quad \text{in }\Omega_R,\\
- \p_1(W^\lambda)_+ + b_{k, 1} (W^\lambda)_+=0 \text{ for } x_1=\pm R.\ea
$$
Using once again an argument similar to the one developed form \eqref{HR}, we deduce that $W^\lambda$ keeps a constant sign on $\Omega_R$. 
Thus 
\beq\label{bdd-W}
\|W^\lambda\|_{L^1(\Omega_R)}= \int W^\lambda=q> 0.
\eeq

We derive an uniform $H^1$ bound on $W^\lambda$ in the following way: we have
$$
\int_{\Omega_R}|\na W^\lambda|^2 - \int_{\Omega_R} b_{k} W^\lambda \cdot \na W^\lambda = \lambda  \int_{\Omega_R}\tilde B_{k,R}(x, W^\lambda) \cdot \na W^\lambda.
$$
Using trace estimates together with the Gagliardo-Nirenberg interpolation as in the proof of Lemma \ref{lem:p}, we have that for any $\nu>0$ there exists $C_\nu>0$ such that
\begin{eqnarray*}
\int_{\Omega_R} b_{k} W^\lambda \cdot \na W^\lambda &=& -\frac{1}{2}\int_{\Omega_R} \dv(b_k) |W^\lambda|^2 + \frac{1}{2}\int_{\bbt^{N-1}}b_{k,1}( R, x')|W^\lambda (R, x')|^2\:dx'\\
&&\quad-\frac{1}{2}\int_{\bbt^{N-1}}b_{k,1}( R, x')|W^\lambda (R, x')|^2\:dx'\\
&\leq & C_\nu \|b_k\|_{W^{1,\infty}} \|W^\lambda\|_{L^1}^{2-\alpha} \|\na W^\lambda\|_2^\alpha + \nu \|\na W^\lambda\|_2^2
\end{eqnarray*}
for some $\alpha\in (0,2)$. On the other hand, setting
$$
\beta_{k,R}(x,r):=\int_0^r \tilde B_{k,R}(x,r')\:dr', \quad x\in \Omega, \ r\in \bbr,
$$
we have (notice that since $\p_v^2 A \in W^{1,\infty}_{loc}$, we also have $\tilde B_{k,R}\in W^{1,\infty}$)
$$
\tilde B_{k,R}(x, W^\lambda) \cdot \na W^\lambda=\dv\left( \beta_{k,R}(x,W^\lambda)\right) - (\dv_x \beta_{k,R})(x,W^\lambda).
$$
Since $\beta_{k,R}(\pm R, x', r)=0$ for all $x', r$,
\begin{eqnarray*}
 \int_{\Omega_R}\tilde B_{k,R}(x, W^\lambda) \cdot \na W^\lambda dx  &=& -\int_{\Omega_R }(\dv_x \beta_{k,R})(x,W^\lambda)dx \\
&=&-\int_{\Omega_R }\int_0^{W^\lambda}(\dv_x \tilde B_{k,R})(x,r')~dr' dx \\
&\le& \int_{\Omega_R }\int_0^{W^\lambda} C~ dr' dx\\
&\leq& C\|W^\lambda\|_{L^1(\Omega_R)}.
\end{eqnarray*}

Using Young's inequality together with the $L^1$ bound \eqref{bdd-W} on $W^\lambda$, we infer that there exists a constant $C$ independent of $\lambda$ and $R$, such that
$$
\|W^\lambda\|_{H^1(\Omega_R)}\leq C.
$$
Therefore, it follows from Schaefer's theorem that \eqref{eq:W-approx} has a solution $W_R\in H^1\cap L^1(\Omega_R)$. Moreover, using the estimates above for $\lambda=1$, we deduce that the family   $(W_R)_{R>0}$ is bounded in $H^1\cap L^1(\Omega_R)$ uniformly in $R$.

Furthermore, we claim that 
\be\label{in:W_R}
0\leq W_R\leq p_{k,R}\quad \forall R.
\ee
The positivity of $W_R$ has been proved above. As for the upper-bound, we notice that  by definition of $\tilde B_{k,R}$, $\tilde B_{k,R}(x, p_{k,R})\equiv 0$, and thus it follows from \eqref{p_kR} that $p_{k,R}$ is a solution of
$$\ba
-\Delta p_{k,R} + \dv(b_kp_{k,R}) +{ \dv \tilde B_{k,R}(x, p_{k,R})}=0\text{ in } \Omega_R,\\
- \p_1 p_{k,R}+ b_{k, 1} p_{k,R}=0 \text{ for } x_1=\pm R.\ea$$
Using the same argument as the one leading to the positivity of $W_R$, we deduce that $W_R-p_{k,R}$ keeps a constant sign over $\Omega_R$. By definition of $k$,
$$
{ q=\int_{\Omega_R} W_R \leq \int_{\Omega} p_k }  
=\int_{\Omega_R} p_{k,R}, 
$$
we deduce that $W_R-p_{k,R}\leq 0$. 

We can now pass to the limit in \eqref{eq:W-approx} as $R\to \infty$. According to the uniform $H^1\cap L^1$ bounds, we deduce that there exists $W\in H^1\cap L^1(\Omega)$ such that $W_R\rightharpoonup W$ in $H^1(\Omega)$, and $W_R\to W$ in $L^2_{loc}(\Omega)$ up to a subsequence. Since $p_{k,R}\to p_k$ in $L^1$ according to Lemma \ref{lem:p}, we deduce that $W$ is a solution of
$$
-\Delta W + \dv(b_kW) + \dv \tilde B_k(x, W)=0.
$$
Eventually, using inequality \eqref{in:W_R} together with the convergence in $L^1$ of the functions $p_{k,R}$, we deduce that $(W_R)_{R>0}$ is uniformly equi-integrable, and therefore compact in $L^1$. Hence, up to a further extraction of subsequences, $W_R\to W$ in $L^1$ and
$$
\int W=\lim_{R\to \infty}\int W_R=q.
$$
Thus the existence of solutions of \eqref{eq:W-2} is proved, which completes the proof of Proposition \ref{prop:shock-mass}.

\section{Stability of stationary shocks} \label{sec:stability}

The goal of this section is to prove Theorem \ref{thm:stability}. Throughout the section, we denote by $(S_t)_{t\geq 0}$ the semi-group associated with equation \eqref{main}. We recall (see for instance \cite{Serre}) that $S_t$ is well-defined in $L^1(\Omega) + L^\infty(\Omega)$, is order preserving and satisfies conservation and contraction principles in $L^1(\Omega)$: if $u,v \in L^1(\Omega) + L^\infty(\Omega)$ are such that $u-v\in L^1(\Omega)$, then $S_t u - S_t v \in L^1(\Omega)$ for all $t\geq 0$ and
$$
\int_{\Omega}(S_t u - S_t v)= \int_{\Omega}(u-v),\quad \| S_t u -S_t v\|_{L^1}\leq \|u-v\|_{L^1}\quad \forall t\geq 0.
$$

First of all, Corollary \ref{cor-zero} allows us to restrict the proof of Theorem \ref{thm:stability} to the case of zero-mass perturbation $u_0\in \bar U+L^1_0(\bbr\times\bbt^{N-1})$.\\

On the other hand, the following lemma allows us to replace the inequality \eqref{hyp:u0} by an inequality where the upper bound and lower bounds are standing shocks.
\begin{lemma}\label{lem-relaxed}
 Let $\bar{U}$ be a stationary shock wave connecting $v(\cdot, p_-)$ to $v(\cdot, p_+)$. Assume that $u\in \bar{U} + L^1_0( \bbr\times\bbt^{N-1})$ satisfies $v(x,p_+) \le u(x)\le v(x,p_-)$ for a.e. $x\in\bbr\times\bbt^{N-1}$.\\
Then, for any $\eps>0$, there exist a function $u^{\eps}\in \bar U+L^1_0(\bbr\times\bbt^{N-1})$ and standing shocks $U^{\eps}_{\pm}$ connecting $v(\cdot, p_-)$ to $v(\cdot, p_+)$ such that 
\[
\|u-u^{\eps}\|_{L^1(\bbr\times\bbt^{N-1})} \le \eps,\quad U^{\eps}_{+}\le u^{\eps}\le U^{\eps}_{-}.
\]
\end{lemma} 
The case of $N=1$ above (i.e., $\bbr$ instead of $\bbr\times\bbt^{N-1}$) has been shown in \cite[Lemma 3.6]{Dalibard-indiana}, whose proof can be directly extended to the above lemma, because other variables $x'$ are in $\bbt^{N-1}$.  The idea is to take $u^\eps=\bar U$ outside of a compact set $[-A_\eps, A_\eps]\times \bbt^{N-1}$ and then to perturb slightly $u$ on the compact  set $[-A_\eps, A_\eps]\times \bbt^{N-1}$ in order to be strictly between the two end states. We leave the details of the proof to the reader since they are identical to \cite[Lemma 3.6]{Dalibard-indiana}.

Now, thanks to Lemma \ref{lem-relaxed} together with the $L^1$-contraction principle, it is enough to prove Theorem \ref{thm:stability} for the class of initial data $u_0\in \bar U+L^1_0(\bbr\times\bbt^{N-1})$ such that
\beq\label{ini-main}
U_+\le u_0 \le U_-,\quad\mbox{for some standing shocks}~ U_{\pm}.
\eeq
Indeed, assume that $\lim_{t\to \infty} \|S_t v_0-\bar U\|_{L^1(\bbr\times\bbt^{N-1})}=0$ for any $v_0\in \bar U+L^1_0(\bbr\times\bbt^{N-1})$ satisfying \eqref{ini-main}. By Lemma \ref{lem-relaxed}, for any $u_0\in \bar U+L^1_0(\bbr\times\bbt^{N-1})$ satisfying \eqref{hyp:u0}, and $\eps>0$, there exists a function $u^{\eps}_0\in \bar U+L^1_0(\bbr\times\bbt^{N-1})$ such that $\|u_0-u^{\eps}_0\|_{L^1(\bbr\times\bbt^{N-1})} \le \eps$ and \eqref{ini-main}. Then the $L^1$-contraction principle yields that for all $t\ge 0$,
\[
 \|S_t u_0-\bar U\|_{L^1(\Omega)} \le\|S_t u_0-S_t u_0^{\eps}\|_{L^1(\Omega)} +\|S_t u_0^{\eps}-\bar U\|_{L^1(\Omega)} \le \eps +\|S_t u_0^{\eps}-\bar U\|_{L^1(\Omega)}.
\]
Since $t\mapsto  \|S_t u_0-\bar U\|_{L^1}$ is non-increasing, and thus has a finite limit as $t\to\infty$,
\[
\lim_{t\to \infty} \|S_t u_0-\bar U\|_{L^1(\bbr\times\bbt^{N-1})}\le \eps,
\]
which implies that $\lim_{t\to \infty} \|S_t u_0-\bar U\|_{L^1(\bbr\times\bbt^{N-1})}=0$.\\

Therefore, there remains to prove Theorem \ref{thm:stability} for the initial data $u_0\in \bar U+L^1_0(\bbr\times\bbt^{N-1})$ satisfying \eqref{ini-main}. We follows the same arguments as \cite{Dalibard-indiana}, which is based on the dynamical system theory due to Osher and Ralston \cite{OR}. The strategy is to prove that the $\omega$-limit set of the trajectory $S_tu_0$ is reduced to $\{\bar U\}$ using the $L^1$-contraction principle. Thus, we need to first show that the $\omega$-limit set is non-empty.\\

$\rhd$\textit{ First step : Structure of the $\omega$-limit set.}

We begin by noticing that the comparison principle together with \eqref{ini-main} imply that for all $t\ge0$,
\[
U_+\le S_t u_0 \le U_-,
\] 
and thus, setting $w(t)= S_t u_0 - \bar U$,
\[
U_+-\bar U\le w(t) \le U_- -\bar U. 
\]
Since $U_+-\bar U$ and $U_- -\bar U$ are in $L^1\cap L^{\infty}(\bbr\times\bbt^{N-1})$ by Lemma \ref{lem-L1}, the family $(w(t))_{t\ge0}$ is equi-integrable in $L^1(\bbr\times\bbt^{N-1})$ and uniformly bounded in $L^{\infty}([0,\infty)\times\Omega)$. Moreover, since $w$ solves a linear parabolic equation of the type
\[
\partial_t w +\dv_x \Big( a(t,x) w\Big) -\Delta w=0,
\] 
where $a:=\int_0^1\partial_vA(x,\tau S_tu_0+ (1-\tau)\bar{U}) d\tau\in L^{\infty}([0,\infty)\times\Omega)$, it follows from \cite[Theorem 10.1]{LSU} that there exists $\alpha>0$ such that for all $t_0\ge1$ and $R>1$,
\[
\|w\|_{H^{\alpha/2,\alpha}((t_0,t_0+1)\times(-R,R)\times\bbt^{N-1})}<\infty.
\]
Thus, $(w(t))_{t\ge0}$ is also equi-continuous in $L^1$.\\
Therefore, it follows from the Riesz-Fr\'echet-Kolmogorov theorem that $(w(t))_{t\ge0}$ is relatively compact in $L^1$. Thus the $\omega$-limit set
\[
\mathcal{B}:=\Big\{ W\in \bar U + L^1_0(\Omega)~|~ \exists (t_n)_{n\in\NN},~t_n\to \infty,~ S_{t_n}u_0\to W~\mbox{in}~L^1(\Omega) \Big\},
\]
is non-empty. Notice that $\mathcal{B}\subset \bar U + L^1_0(\Omega)$ because of $u_0\in \bar U + L^1_0(\Omega)$ and the conservation of mass.\\
By the definition of $\omega$-limit set, $\mathcal{B}$ is forward and backward invariant by the semi-group $S_t$, i.e., $S_t\mathcal{B}=\mathcal{B}$ for all $t$. Moreover, thanks to parabolic regularity, all functions in $\mathcal{B}$ are smooth, for example $\mathcal{B}\subset H^1_{loc}(\Omega)$. As a consequence, for any $W\in \mathcal{B}$, it follows from \cite[Theorem 6.1]{LSU} that $S_tW\in L^2(0,T; H^2_{loc}(\Omega))\cap H^1(0,T; L^2_{loc}(\Omega))$.

We take advantage of LaSalle invariance principle \cite{L} with a suitable Lyapunov functional $F[u]:=\|u-\bar U\|_{L^1(\Omega)}$. Since $t\mapsto F[S_t W]$ is non-increasing by the $L^1$-contraction principle, $F$ takes a constant value on $\mathcal{B}$, which we denote by $C_0$. \\

$\rhd$\textit{ Second step :  $\mathcal{B}=\{\bar U\}$.}
We now prove $\mathcal{B}=\{\bar U\}$. For any $W_0\in \mathcal{B}$, we set $W(t)=S_tW_0$. Notice that $W(t)\in\Omega$ for all $t\ge 0$. Since $W(t)-\bar U$ satisfies 
\[
\partial_t (W-\bar U) + \dv_x \Big( A(x,W)-A(x,\bar U) \Big) -\Delta (W-\bar U)=0,
\]
we have
\beq\label{eq-abs}
\partial_t |W-\bar U| + \dv_x \Big( b(t,x) |W-\bar U| \Big) - \sgn(W-\bar U)\Delta (W-\bar U)=0,
\eeq
where $b(t,x)=\int_0^1\partial_vA(x,\tau W+ (1-\tau)\bar{U}) d\tau$.\\
In order to show that $\sgn(W-\bar U)\Delta (W-\bar U)=\Delta |W-\bar U|$, we use the following lemma.
\begin{lemma}\label{lem-B.1}
Let $f\in L^1\cap L^{\infty} (\bbr\times \bbt^{N-1})$ such that $\nabla f\in L^2(\bbr\times \bbt^{N-1})$ and $\Delta f\in L^1_{loc}(\bbr\times \bbt^{N-1})$. Assume that $f$ satisfies 
\beq\label{condition-1}
\lim_{R\to\infty} \int_{\Omega} \sgn(f) \Delta f\, \theta\Big(\frac{x_1}{R}\Big) dx =0,
\eeq
for all $\theta\in \mathcal{C}^{\infty}_0 (\bbr)$ such that $\theta\equiv 1$ in a neighborhood of the origin. Then
\[
\lim_{\delta\to0}\frac{1}{\delta} |\nabla f|^2  \mathbf 1_{|f|<\delta} =0 \quad \mbox{in}~\mathcal{D}'(\Omega),
\]
therefore,
\[
\sgn(f)\Delta f=\Delta |f| \quad \mbox{in}~\mathcal{D}'(\Omega).
\]
\end{lemma}
The case of $N=1$ above has been shown in \cite[Lemma B.1]{Dalibard-indiana}, whose proof can be directly extended to the above lemma. Now, in order to show that the condition \eqref{condition-1} is satisfied in our case, we recall from the previous step that $F[W(t)]=\|W(t)-\bar U\|_{L^1(\Omega)}=C_0$ for all $t\ge 0$.
For any $t'>t\ge0$, since 
\begin{eqnarray*}
\int_t^{t'}\int_{\Omega} \partial_t |W-\bar U|\theta(\frac{x_1}{R})dx ds &\le & \int_{\Omega} |W(t')-\bar U| dx-\int_{|x_1|\le CR} |W(t)-\bar U| dx \\
&=& \int_{|x_1|\ge CR} |W(t)-\bar U| dx \to 0\quad\mbox{as}~R\to\infty,
\end{eqnarray*}
and 
\begin{eqnarray*}
\int_t^{t'}\int_{\Omega} \dv_x \Big( b(t,x) |W-\bar U| \Big) \theta(\frac{x_1}{R})dx ds 
&\le &\|A\|_{W^{1,\infty}}\|\theta^{\prime}\|_{\infty}\frac{1}{R} \int_t^{t'}\int_{\Omega}|W(t)-\bar U| dx ds \\
&=& \frac{C(t'-t)}{R} \to 0\quad\mbox{as}~R\to\infty,
\end{eqnarray*}
we have
\[
\int_t^{t'}\int_{\Omega}\sgn(W-\bar U)\Delta (W-\bar U)\theta(\frac{x_1}{R}) dxds \to 0\quad\mbox{as}~R\to\infty.
\]
Thus, a slightly modified version of Lemma \ref{lem-B.1} implies that
\[
\sgn(W-\bar U)\Delta (W-\bar U)=\Delta |W-\bar U|.
\]
Therefore, $|W-\bar U|$ is a non-negative solution of a parabolic equation of the type
\[
\partial_t |W-\bar U| + \dv_x \Big( b(t,x) |W-\bar U| \Big) - \Delta |W-\bar U|=0,
\]
where $b\in L^{\infty}([0,\infty)\times\Omega)$. Thanks to the Harnack inequality for the parabolic equations, for any compact set $K$ in $\Omega$, there exists $C_K$ such that
\beq\label{H-eq}
\sup_{x\in K} |(W_0-\bar U)(x)| \le C_K \inf_{x\in K} |(W(1)-\bar U)(x)|.
\eeq
Moreover, using the fact that $W(1)-\bar U\in L^1_0\cap H^1_{loc}(\Omega)$, there exists $x_1\in\Omega$ such that
\[
(W(1)-\bar U)(x_1) =0,
\]
which implies together with \eqref{H-eq} that $W_0\equiv V$. Hence we have $\mathcal{B}=\{\bar U\}$, and thus complete the proof of Theorem \ref{thm:stability}.

\begin{appendix}
\setcounter{equation}{0}
\section{use of the Krein-Rutman theorem to prove the positivity of solutions of some elliptic equations}

In this Appendix, we prove the following result, which has been used in several instances in the paper:

\begin{lem}\label{lem-app}
Let $R>0$ be arbitrary, and let $b\in L^\infty(\Omega_R)$. Consider the equation
\be\label{w-KR}
\ba
-\Delta w + \dv (bw)=0\quad \text{in } \Omega_R,\\
-\p_1 w + b_1 w=0\text{ for } x_1=\pm R.
\ea
\ee
Then the vector space of solutions of equation \eqref{w-KR} is $\bbr w_1$, where $w_1\in H^1(\Omega_R)\cap\mathcal C (\bar \Omega_R)$ is a strictly positive solution of \eqref{w-KR} such that $\int_{\Omega_R}w_1=1$.

\end{lem}

\begin{proof}
The dual of problem \eqref{w-KR} is
$$
\ba
- \Delta q - b\cdot \na q=0\quad \text{in } (-R,R)\times \bbt^{N-1},\\
\p_1 q=0\text{ at } x_1=\pm R,
\ea
$$
of which the constant function equal to one is a strictly positive solution.

Let us introduce the operator $F : u\in L^2(\Omega_R)\mapsto v\in L^2(\Omega_R)$
where $v=F(u)$ is the unique solution of the equation
$$
-\Delta v -b\cdot \nabla v + \alpha v=\alpha u\ \text{in }\Omega_R,\quad \p_1 v=0\text{ at } x_1=\pm R,
$$
and $\alpha$ is a positive constant chosen so that the bilinear
form associated to $F$ is coercive (e.g.
$\alpha=\frac{||b||_{\infty}^2}{2} + \frac{1}{2}$). With that
choice of $\alpha$, $F$ is a strictly positive operator.

Next, using  regularity results for linear
elliptic equations, we show that $F$
maps $L^q(\Omega_R)$ into $W^{2,q}(\Omega_R)$ for all $q\geq 2$.
Hence, the restriction of $F$ to $\mathcal
C(\bar \Omega_R)$, still denoted by $F$, is a compact
operator from $\mathcal
C(\bar \Omega_R)$ into itself.
The last step consists in using the strong form of the maximum principle together with Hopf's Lemma: if $u\in
\mathcal
C(\bar \Omega_R)$, $u\geq 0$, $u\neq 0$ and
$v=F(u)$, then $v(y)>0$ for all $y\in \bar \Omega_R$.

Hence, $F:\mathcal
C(\bar \Omega_R)\to\mathcal
C(\bar \Omega_R)$ is a strongly positive operator.
We conclude by using the strong form of the Krein-Rutman theorem
(see \cite{DL,KR}): since $F(\bar{1})=\bar{1}$, the
spectral radius of $F$ is equal to 1 and $1$ is a simple
eigenvalue of $F^*$, the adjoint of $F$, with a positive
eigenvector. We infer that \eqref{w-KR} has a unique non-negative solution $w_1$ normalized in  $L^1$. 

\end{proof}

\end{appendix}

\bibliographystyle{amsplain}
\bibliography{bibliografie}

\end{document}